\documentclass{article}

\usepackage{etoolbox}

\newtoggle{for_arxiv}

\toggletrue{for_arxiv}

\nottoggle{for_arxiv}{%
    \usepackage[preprint]{neurips_2019}
} {
    \usepackage[authoryear]{natbib}
}

\usepackage{microtype}
\usepackage{graphicx}
\usepackage{subfigure}
\usepackage{booktabs} 
\usepackage{xcolor}
\usepackage[hidelinks=True]{hyperref}
\usepackage{xargs}[2008/03/08]

\usepackage{refstyle}
\usepackage{varioref} 

\usepackage{amsmath}
\usepackage{amssymb}
\usepackage{amsfonts}
\usepackage{amsthm}
\usepackage{mathrsfs} 
\usepackage{mathtools}

\usepackage{algorithm}
\usepackage{algpseudocode}
\usepackage{listings}
\usepackage{pdfpages}

\usepackage[]{graphicx}
\usepackage[]{color}
\makeatletter
\def\maxwidth{ %
  \ifdim\Gin@nat@width>\linewidth
    \linewidth
  \else
    \Gin@nat@width
  \fi
}
\makeatother

\definecolor{fgcolor}{rgb}{0.345, 0.345, 0.345}

\usepackage{framed}
\makeatletter
 {\par\unskip\endMakeFramed%
 \at@end@of@kframe}
\makeatother

\definecolor{shadecolor}{rgb}{.97, .97, .97}
\definecolor{messagecolor}{rgb}{0, 0, 0}
\definecolor{warningcolor}{rgb}{1, 0, 1}
\definecolor{errorcolor}{rgb}{1, 0, 0}

\usepackage{alltt}

\def\mbe{\mathbb{E}}%
\def\cov{\mathrm{Cov}}%

\def\norm#1{\left\Vert #1\right\Vert }
\def\derivat#1#2#3{\left.\frac{d#1}{d#2}\right\vert_{#3}}
\def\partialat#1#2#3{\left.\frac{\partial#1}{\partial#2}\right\vert_{#3}}
\def\fwd{\textsc{ForwardModeAD}}

\newcommandx\Gk[1][usedefault, 1=k]{d^{#1}_\theta G}
\newcommandx\Gkw[1][usedefault, 1=k]{\delta_w d^{#1} G}
\newcommandx\gnk[2][usedefault, 1=k, 2=n]{d^{#1}_\theta g_{#2}}
\newcommandx\gk[1][usedefault, 1=k]{d^{#1}_\theta g}
\newcommandx\dtheta[2][usedefault, 1=k, 2=\wt]{\delta^{#1}_w \hat\theta(#2)}
\newcommandx\Hw[1][usedefault, 1=\wt]{H(#1)}
\def\Hone{\hat{H}}
\def\onevec{1_{N}}
\def\deltaw{\Delta w}
\def\wdom{\Omega_w}
\def\wepsilon{\varepsilon}
\def\allthetadom{\Omega_\theta}
\def\allthetadom{\Omega_\theta}

\def\allwdom{\mathscr{W}}

\def\kij{k_{\mathrm{IJ}}}%
\newcommandx\deltak[1][usedefault, 1=k]{\delta_{#1}}
\def\deltamax{\delta}
\def\cop{C_{op}}%
\def\cset{C_{set}}%
\newcommandx\mk[1][usedefault, 1=k]{M_{#1}}
\def\lh{\mk[3]} 
\def\deltabound{\rho}%
\def\copw{\tilde{C}_{op}}
\newcommandx\dthetaconst[1][usedefault, 1=k]{C_{\Delta;#1}}
\newcommandx\ijconst[1][usedefault, 1=k]{C_{\mathrm{IJ};#1}}

\newcommandx\thetaw[1][usedefault, 1=\wt]{\hat\theta(#1)}
\def\thetaone{\hat\theta}
\def\wt{\tilde{w}}
\newcommandx\thetaij[2][usedefault, 1=k, 2=w]{\hat\theta^{\mathrm{IJ}}_{#1}(#2)}
\newcommandx\errij[1][usedefault, 1=k]{\varepsilon^\mathrm{IJ}_{#1}}
\def\kvec{\mathbb{K}}%
\newcommandx\dterm[3][usedefault, 1=\kvec, 2=\omega, 3=\wt]
    {\tau \left(#1, #2, #3\right)}
\newcommandx\dtermargs[1][usedefault, 1=k]{\mathscr{T}_{#1}}
\newcommandx\dthetaset[1][usedefault, 1=k]{\Theta_{#1}}

\def\vk{V_k}
\newcommandx\tk[2][usedefault, 1=k, 2=\nset]{T_{#1}(#2)}
\newcommandx\gsupnk[2][usedefault, 1=k, 2=n]{\gamma_{#2}^{(#1)}}
\newcommandx\gsupk[1][usedefault, 1=k]{\gamma^{(#1)}}
\def\nset{\mathscr{N}}
\newcommandx\quant[2][usedefault, 1=k, 2=\alpha]{q_{#2}^{(#1)}}
\newcommandx\cvwdom[1][usedefault, 1=\kappa]{\allwdom_{CV;#1}}
\def\loodom{\allwdom_{\mathrm{loo}}}

\def\gimp{\tilde{G}}
\def\zerovec{0_D}
\def\zdom{\Omega_z}

\newcommandx\fk[1][usedefault, 1=k]{f^{(#1)}}
\newcommandx\fdk[1][usedefault, 1=k]{f_d^{(#1)}}
\def\deltax{\left(x_1 - x_0\right)}%
\def\xt{x(t)}%

\theoremstyle{plain}
\newtheorem{lem}{Lemma}
\newtheorem{thm}{Theorem}
\newtheorem{prop}{Proposition}
\newtheorem{cond}{Condition}
\newtheorem{assu}{Assumption}
\newtheorem{cor}{Corollary}

\theoremstyle{definition}
\newtheorem{defn}{Definition}

\newref{sec}{
    name=Section~, %
    Name=Section~
    }

\newref{app}{
    name=Appendix~, %
    Name=Appendix~
    }

\newref{eq}{
    name=Eq.~, %
    Name=Eq.~,
    names=Eqs.~, %
    Names=Eqs.~
    }

\newref{fig}{
    name=Figure~, %
    Name=Figure~
    }

\newref{def}{
    name=Definition~, %
    Name=Definition~
    }

\newref{assu}{
    name=Assumption~, %
    Name=Assumption~,
    names=Assumptions~, %
    Names=Assumptions~,
    }

\newref{cond}{
    name=Condition~, %
    Name=Condition~,
    names=Conditions~, %
    Names=Conditions~
    }

\newref{prop}{
    name=Proposition~, %
    Name=Proposition~,
    names=Propositions~, %
    Names=Propositions~
    }

\newref{lem}{
    name=Lemma~, %
    Name=Lemma~
    }

\newref{ex}{
    name=example~, %
    Name=Example~}

\newref{cor}{
    name=Corollary~, %
    Name=Corollary~
    }

\newref{thm}{
    name=Theorem~, %
    Name=Theorem~
    }

\newref{proof}{
    name=Proof~, %
    Name=Proof~
    }

\newref{conj}{
    name=Conjecture~, %
    Name=Conjecture~
    }

\newref{algr}{
    name=Algorithm~, %
    Name=Algorithm~,
    names=Algorithms~, %
    Names=Algorithms~,
    }

\title{A Higher-Order Swiss Army Infinitesimal Jackknife}

\author{
  Ryan Giordano 
  \and
  Michael I.~Jordan 
  \and
  Tamara Broderick 
}

\begin{document}


\newcommand{\gonlyassums}{\assurangeref{thetadom}{bounded}}
\newcommand{\gassums}{\assurangeref{thetadom}{complexity}}
\newcommand{\wconds}{\condref{theta_small}}
\newcommand{\allassums}{\gassums{} and \wconds{}}
\newcommand{\seeproof}[1]{(Proof \proofpageref[vref]{#1}.)}

\maketitle

\begin{abstract}
Cross validation (CV) and the bootstrap are ubiquitous model-agnostic tools for
assessing the error or variability of machine learning and statistical
estimators. However, these methods require repeatedly re-fitting the model with
different weighted versions of the original dataset, which can be prohibitively
time-consuming.  For sufficiently regular optimization problems the optimum
depends smoothly on the data weights, and so the process of repeatedly
re-fitting can be approximated with a Taylor series that can be often evaluated
relatively quickly.  The first-order approximation is known as the
``infinitesimal jackknife'' in the statistics literature and has been the
subject of recent interest in machine learning for approximate CV. In this work,
we consider high-order approximations, which we call the ``higher-order
infinitesimal jackknife'' (HOIJ). Under mild regularity conditions, we provide a
simple recursive procedure to compute approximations of all orders with
finite-sample accuracy bounds. Additionally, we show that the HOIJ can be
efficiently computed even in high dimensions using forward-mode automatic
differentiation.  We show that a linear approximation with bootstrap weights
approximation is equivalent to those provided by asymptotic normal
approximations.  Consequently, the HOIJ opens up the possibility of enjoying
higher-order accuracy properties of the bootstrap using local approximations.
Consistency of the HOIJ for leave-one-out CV under different asymptotic regimes
follows as corollaries from our finite-sample bounds under additional regularity
assumptions.  The generality of the computation and bounds motivate the name
``higher-order Swiss Army infinitesimal jackknife.''

\end{abstract}

\section{Introduction}\label{sec:introduction}
When statistical machine learning models are employed in real-world settings it
is important to assess their error and variability.  Unavoidable randomness in
data collection and the need for our model to extrapolate well to data not yet
observed motivates the question: ``How would my inference have changed had I
observed different data?''  To answer this question when we are not confident in
the truth of our stochastic model but are only willing to assume that our data
is exchangeable we turn to model-agnostic procedures.  Typically, such
procedures rely on re-fitting the model with modified versions of the original
dataset.  Instances of this framework are cross-validation (CV), which leaves
out a certain number of data points, and the bootstrap, which uses a sample of
data points drawn with replacement from the original dataset.  Though appealing
for their simplicity and generality, the fact that the CV and bootstrap require
repeatedly re-fitting the model means that they are computationally intensive,
sometimes prohibitively so.

We propose approximating the dependence of a fitted model on a modified
dataset by a higher-order Taylor series approximation.  We call this Taylor
series approximation a ``higher-order infinitesimal jackknife'' (HOIJ), as the
first-order version of the Taylor series is known in the statistical literature
as the ``infinitesimal jackknife''.  The ``jackknife'' was originally so named
because ``[it] refers to the jack-of-all-trades and the master of none, and is
intended to denote a tool which has a very wide application but which can in
most cases be improved upon by more special-purpose equipment''
\citep{barnard:1974:cvchoicediscussion}.  We aim to similarly provide
a fast, usable, general-purpose perturbative approximation to CV and bootstrap.
To this end, we provide finite-sample error bounds for the difference between
the HOIJ and exact refitting under assumptions that can be verified with the
dataset at hand, without recourse to asymptotics or stochastic assumptions.
Nevertheless, under reasonable additional assumptions, our results also provide
asymptotic results that match the existing literature for first-order
approximations.

We additionally show that the HOIJ can be efficiently computed even in high
dimensions without storing large derivative arrays in memory.
Evaluating the approximation for a $D$-dimensional parameter requires first
calculating and factorizing a single $D\times D$ matrix, but, once this is done,
evaluating the Taylor series with each new set of weights requires only
matrix-vector and derivative-vector products, which can be computed easily and
efficiently using automatic differentiation.  Due to the generality of our
methods, we propose that the HOIJ adds extra tools to the ``Swiss Army
infinitesimal jackknife'' \citep{giordano:2019:ij}.

\section{Setup and notation}\label{sec:notation}
We will consider estimators $\hat\theta(w) \in \mathbb{R}^D$ that are solutions
to the following system of equations:
\begin{align}\eqlabel{g_definition}
\thetaw[w] := \theta\quad\textrm{such that}\quad
G(\theta, w) :=
    \frac{1}{N} \left(g_0(\theta) + \sum_{n=1}^N w_n g_n(\theta)\right) = 0.
\end{align}
$G(\theta, w)$ and $g_n(\theta)$ are vectors of length $D$, and $w=(w_1, \ldots,
w_N)$ is a vector of length $N$. The term $g_0(\theta)$ can be thought of as a
prior or regularization term, whereas $g_n(\theta)$ for $n>1$ can be thought of
as depending on the data. \Eqref{g_definition} arises, for example when finding
a local solution to an smooth optimization problem: suppose that $\thetaw[w]$ is
chosen to solve $\thetaw[w] := \mathrm{argmin}_{\theta} \sum_{n=1}^N w_n
f(\theta, x_n) + f_0(\theta)$ for some dataset $\{x_n: n=1,\ldots,N\}$, and
smooth scalar-valued functions $f(\theta, x_n)$ and $f_0(\theta)$.  Then a local
optimum has the form of \eqref{g_definition} with $g_n(\theta) := \partial
f(\theta, x_n) / \partial \theta$ and $g_0(\theta) := \partial f_0(\theta) /
\partial \theta$.  As discussed in \citet{giordano:2019:ij},
\eqref{g_definition} also encompasses sequences of optimization problems in
which the solution to one optimization problem is used as a hyperparameter in
another.

We will assume that we have found the solution of \eqref{g_definition} for the
original dataset, i.e., with $w$ equal to a length-$N$ vector containing all
ones.  Denote the all-ones weight vector $\onevec$, and the corresponding
estimate as $\thetaone := \thetaw[\onevec]$. Performing CV or the bootstrap
amounts to evaluating $\thetaw[w]$ with different weights.  For example,
leave-$\kappa$-out CV would use weights with exactly $\kappa$ zeros and the rest
ones, $\ell$-fold cross validation would use weights with $\lfloor N / \ell
\rfloor$ zeros and the rest ones, and the bootstrap would use weights drawn from
a multinomial distribution with count $N$, $N$ outcomes, and probabilities
$N^{-1}$.

Each weight vector requires a new solution of \eqref{g_definition}, which can be
computationally expensive.  Rather than solving \eqref{g_definition} exactly, we
might approximate $\thetaw[w]$ with a Taylor series expansion in $w$ centered at
$\onevec$.  The first-order Taylor series approximation is known as the
``infinitesimal jackknife'' due to its relationship to the jackknife bias and
variance  estimator \citep{jaeckel:1972:infinitesimal}.  In this work, we
provide tools to compute and analyze the \emph{higher-order infinitesimal
jackknife} (HOIJ), by which we simply mean higher-order Taylor series expansions
of $\thetaw[w]$.

In order to define and analyze the HOIJ, we will need to represent and
manipulate higher-order derivatives, and for this it will be helpful to
introduce some compact notation.  Let $\deltaw := w - \onevec$ and, as with $w$,
denote the $n$-th element of $\deltaw$ as $\deltaw_n$.   For a particular $w$,
we will be considering directional derivatives of $\thetaw$ in the direction
$\deltaw$ evaluated at some (possibly different) weight vector $\wt$.  We denote
such a directional derivative with the following shorthand:
\begin{align*}
    \dtheta[1] := \sum_{n=1}^N  \partialat{\thetaw}{w_n}{\wt} \deltaw_n
    \textrm{,}\quad\quad
    \dtheta[2] := \sum_{n_1=1}^N \sum_{n_2=1}^N
        \partialat{{}^2\thetaw}{w_{n_1} \partial w_{n_2}}{\wt}
            \deltaw_{n_1} \deltaw_{n_2},
\end{align*}
and so on, where the $k$-th order directional derivative is denoted $\dtheta$.
We will only be considering directional derivatives of $\thetaw$ in the direction
$\deltaw$, so, to avoid clutter, the direction $\deltaw$ is left implicit
in the notation $\dtheta$.

For some fixed integer $\kij \ge 0$, we will evaluate and analyze a $\kij$-th
order Taylor series centered at $\onevec$ and evaluated at $w$. In our notation,
this Taylor series and its error can be written as follows.
%
%
\begin{defn}\deflabel{ij}
%
For a fixed $\kij \ge 0$, when $\dtheta[k][\onevec]$ exists for all
$k \le \kij$, define
\begin{align*}
\thetaij[\kij] :=
    \thetaone + \sum_{k=1}^{\kij} \frac{1}{k!} \dtheta[k][\onevec]
\quad\quad\textrm{and}\quad\quad
\errij[\kij](w) := \norm{\thetaij[\kij] - \thetaw[w]}_2.
\end{align*}
\end{defn}
%
%
If $\kij = 0$, we take the empty sum to be $0$ and $\thetaij[0] = \thetaone$.
As with any Taylor series, the size of the error, $\errij[\kij](w)$, will depend
on the magnitude of the next higher derivative, $\dtheta[\kij + 1][\wt]$ (when
it exists), evaluated at intermediate points $\wt: \norm{\wt - \onevec}_2 \le
\norm{w - \onevec}_2$.

The dependence of $\thetaw[w]$ on $w$ is implicit through the estimating
equation $G(\theta, w) = 0$, so the derivatives in the definition of
$\thetaij[\kij]$ cannot be evaluated directly.  One of our contributions, which
we describe in \secref{taylor_expansion}, is to develop a recursive procedure
to evaluate $\dtheta[][w]$ in terms of higher-order partial directional
derivatives of $G(\theta, w)$ that are easy to compute and analyze.  We
now introduce notation for such derivatives.

Let $\Gk(\theta, w)$ represent the $D^{k + 1}$-dimensional array of partial
derivatives $\partial^k G(\theta, w) / \partial\theta \ldots \partial\theta$.
Similarly, $\gnk(\theta)$ denotes the $k$-th partial derivative of the
individual term $g_n(\theta)$, and, by $\gk$, we denote the length-$N + 1$ array
of derivative arrays $\gk := \left(\gnk[k][0], \gnk[k][1], \ldots,
\gnk[k][N]\right)$.  It will be convenient to write $\Gk[0](\theta, w) =
G(\theta, w)$, with a corresponding definition for $\gnk[0](\theta)$ and
$\gk[0](\theta)$.
We denote $k$-th order partial directional derivatives of $G(\theta, w)$ with
respect to $\theta$ in the directions $v_1,\ldots,v_k$ (for $v_1, v_2, \ldots
\in \mathbb{R}^D$) by the shorthand
\begin{align*}
\Gk(\theta, \wt) v_1 \ldots v_k :=
    \sum_{d_1=1}^{D} \ldots \sum_{d_k=1}^{D}
    \partialat{{}^k G(\theta, w)}
              {\theta_{d_1} \ldots \partial\theta_{d_k}}
              {\theta, \wt}
        v_{1, d_1} \ldots v_{k, d_k}.
\end{align*}
To motivate the notation for the directional derivatives of $G(\theta, w)$, one
can think of $\Gk(\theta, w)$ as a map from a direction $v$ to a $k-1$-st order
derivative.  Then, for example, $\Gk[2](\theta, \wt)v_1 v_2$ is shorthand for
$\left(\left(\Gk[2](\theta, \wt)\right)(v_1)\right)(v_2)$ --- first, the
second-order derivative is evaluated at $(\theta, \wt)$, then applied to $v_1$,
then the resulting first-order derivative is applied to $v_2$, finally resulting
in a vector in $\mathbb{R}^D$.

Since $G(\theta, w)$ is linear in $w$, only the first partial derivative with
respect to $w$ is non-zero, and this partial derivative does not depend on the
weight vector at which it is evaluated.  Furthermore, like the derivatives of
$\thetaw$, we will only be considering directional derivatives of
$G(\theta, w)$ in the direction $\deltaw$.  We can thus avoid introducing
general notation for partial derivatives of $G(\theta, w)$ with respect to
$w$ and compactly write
\begin{align*}
\Gkw(\theta) &:=
    \partialat{\left(\Gk(\theta, w)\right)}{w}{\wt} \deltaw.
\end{align*}
Since $g_0(\theta)$ does not depend on $w$ and $G(\theta, w)$
is linear in $w$, note that
\begin{align*}
\Gkw(\theta)
&= \frac{1}{N}  \partialat{\left(\gnk[][0](\theta)\right)}{w_n}{\wt} +
    \frac{1}{N} \sum_{n=1}^{N}
        \partialat{\left(\gnk(\theta)\right) w_n }{w_n}{\wt} \deltaw_n \\
&= \frac{1}{N} \sum_{n=1}^{N} \gnk(\theta) \deltaw_n \\
&= \Gk(\theta, w) - \Gk(\theta, \onevec).
\end{align*}

The matrix of first partial derivatives $\Gk[1](\theta, w)$ will play a special
role, and so we give it a distinct symbol:
\begin{align*}
\Hw[\theta, w] := \Gk[1](\theta, w)
    \quad\quad\textrm{and}\quad\quad
\Hone := \Hw[\thetaone, \onevec].
\end{align*}
We use the letter ``H'' because $\Hw[\theta, w]$ is the Hessian matrix of the
objective when the estimating equation $G(\theta, w)$ arises from a smooth
optimization problem. Finally, we adopt the shorthand that $\Gk(\theta, w)$ and
$\Hw[\theta, w]$ without a $\theta$ argument are taken to be evaluated at
$\thetaw$:
\begin{align*}
    \Gk(\wt) := \Gk(\thetaw, \wt)\quad\textrm{and}\quad
    \Hw[\wt] := \Hw[\thetaw, \wt].
\end{align*}
We will use the norms $\norm{\cdot}_1$, $\norm{\cdot}_2$, and
$\norm{\cdot}_\infty$ to refer to the ordinary $L_p$ norm of the vectorized
version of whatever quantity appears inside the norm.  For example,
\begin{align*}
\norm{\Hw}_2 =
    \sqrt{\sum_{d_1 = 1}^D \sum_{d_2 = 1}^D \left(\Hw_{d_1 d_2}\right)^2},
\end{align*}
and
\begin{align*}
\norm{\gk[1]}_2 = \sqrt{\sum_{n=0}^N \sum_{d_1 = 1}^D \sum_{d_2 = 1}^D
(\gnk[1](\theta)_{d_1 d_2})^2}.
\end{align*}
We will make use of the well-known fact that $\norm{\cdot}_\infty \le
\norm{\cdot}_2 \le \norm{\cdot}_1$. Note that, by Jensen's inequality,
$\norm{\Gk(\theta)}_p \le \frac{1}{N} \norm{\gk(\theta)}_p$; see
\propref{array_norm_relation} in the appendix.
For matrices, we use $\norm{\cdot}_{op}$ to represent the operator norm,
i.e., the norm dual to the vector norm $\norm{\cdot}_2$.  Recall that,
if $A$ is a matrix, $\norm{A}_{op} \le \norm{A}_2$.


\section{Computing the Taylor series}\seclabel{taylor_expansion}
The estimator $\thetaw[w]$ is implicitly a function of $w$ through the
estimating equation $G(\theta, w) = 0$, so we cannot directly compute the
derivatives necessary to compute $\thetaij$.  In this section, we resolve this
difficulty and derive recursive expressions for $\dtheta$ in terms of the
partial derivatives $\Gk(\theta, \wt)$ and lower-order derivatives
$\dtheta[k']$ for $k' < k$. For the moment, we will assume that sufficient
regularity conditions hold that all our expressions are well-defined.  In
particular, we assume for the moment that $\thetaw[w]$ is a well-defined and
continuously differentiable function of $w$. We will precisely state sufficient
regularity conditions in \secref{assumptions} below.

Recall our shorthand notation  $\Gk(\wt) := \Gk(\thetaw, \wt)$.  In order to
derive expressions for $\dtheta$, we will repeatedly take
directional derivatives of $G(\wt)$ in the direction $\deltaw$. By
definition, $G(\wt) = 0$ for all $\wt$.  Differentiating this expression
using the chain rule gives
\begin{align}\eqlabel{dg1}
0  &=
\derivat{G(w)}{w}{\wt} \deltaw \nonumber \\
    &= \partialat{G(\theta, w)}{\theta}{\thetaw, \wt}
        \derivat{\thetaw[w]}{w}{\wt} \deltaw +
        \partialat{G(\theta, w)}{w}{\thetaw, \wt} \deltaw \nonumber \\
    &= \Gk[1](\wt) \dtheta[1] + \Gkw[0](\thetaw).
\end{align}
Recall that, by convention, we take $\Gk[0](\thetaw) = G(\thetaw)$.  This
notation will be convenient later when we develop recursive expressions for
higher-order derivatives. By evaluating \eqref{dg1} at $\wt = \onevec$ and
assuming that the matrix $\Gk[1](\onevec) = \Hone$ is invertible, we can solve
for $\dtheta[1][\onevec]$ and so derive the \emph{first-order infinitesimal
jackknife} (see \defref{ij}):
\begin{align}\eqlabel{dtheta1}
\dtheta[1][\onevec] = -\Hone^{-1} \Gkw[0](\thetaone)
    \quad\quad\textrm{and}\quad\quad
\thetaij[1] := \thetaone + \dtheta[1][\onevec].
\end{align}
Higher-order derivatives $\dtheta$ can be evaluated by continuing to
differentiate \eqref{dg1} term by term using the chain rule:
\begin{align}\eqlabel{dg2}
0 &=
    \derivat{{}^2 G(w)}{w dw}{\wt} \deltaw \deltaw \nonumber \\
&= \Gk[1](\wt) \dtheta[2] + {}
    & \textrm{(differentiating }\dtheta[1]\textrm{)} \nonumber \\
& \quad \Gk[2](\wt) \dtheta[1] \dtheta[1] +  {}
    & \textrm{(differentiating }\Gk[1](\theta, \wt)
        \textrm{ through }\theta\textrm{)} \nonumber \\
& \quad \Gkw[1](\wt) \dtheta[1] +  {}
    & \textrm{(differentiating }\Gk[1](\theta, w)
        \textrm{ through }w\textrm{)} \nonumber \\
& \quad \Gkw[1](\wt) \dtheta[1].
    & \textrm{(differentiating }\Gkw[0](\thetaw)
        \textrm{ through }\theta\textrm{)}
\end{align}
Observe that $\dtheta[2]$ occurs exactly once in \eqref{dg2}, again multiplied
by $\Gk[1](\wt)=\Hw$, so we can again evaluate at $\onevec$ and solve for
$\dtheta[2]$:
\begin{align*}
\dtheta[2] &= -\Hw^{-1}
    \left(\Gk[2](\wt) \dtheta[1] \dtheta[1] +
          2 \Gkw[1](\thetaw)\dtheta[1] \right).
\end{align*}

The formula for $\dtheta[2]$ is useful for at least two reasons.  First,
we can evaluate at $\onevec$ to form the
\emph{second-order infinitesimal jackknife} (as given in \defref{ij}).
Second, we can use the formula for $\dtheta[2]$ evaluated at $\wt$ in the set $\{
\wt : \norm{\wt - \onevec}_2 \le \norm{w - \onevec}_2 \}$ to control the error
of Taylor series used in the definition of the first-order infinitesimal
jackknife, $\thetaij[1]$. Under what conditions might we expect
$\norm{\dtheta[2]}_2$ (and so the error in $\thetaij[1]$) to be small uniformly
in $\wt$? Intuitively, we might hope that (a) the directional derivatives
$\dtheta[1]$ and $\Gkw[1](\thetaw)$ may be small, and (b) quantities like
$\Gk[1](\wt)$ and $\Gk[2](\wt)$ are close to their respective counterparts
$\Gk[1](\onevec)$ and $\Gk[2](\onevec)$.  Note that $\Gk[1](\onevec)$ and
$\Gk[2](\onevec)$ are simply sample averages of partial derivatives evaluated at
the original data and may reasonably be bounded by assumption. In
\secref{assumptions} we state precisely conditions under which this intuition
holds, and our proofs all proceed essentially in this way.

Analogously, we can continue the process and differentiate \eqref{dg2} to form a
\emph{third-order infinitesimal jackknife} and to control the error in
$\thetaij[2]$, and so on, for up to a $\kij$-order infinitesimal jackknife,
whose accuracy bound will require a $\kij + 1$-th order derivative.
In order to analyze all orders in this way, it will be useful to define a
general representation of terms arising from the use of the product rule as in
\eqref{dg1} and \eqref{dg2}.  To that end, we observe that every term that
arises from continued application of the product rule to differentiate
$G(\wt)$ will have the following form.

\begin{defn}\deflabel{deriv_terms}
%
Let $\kvec$ be a size $|\kvec|$ (possibly empty)
set of positive integers. Let $\omega\in\left\{ 0, 1 \right\}$.  Let $\wt$ be a
location at which the derivative is evaluated.  Then let $\dterm$ denote any
term of the form
\begin{align*}
\dterm &:= \begin{cases}
\Gk[|\kvec|](\wt) \prod_{k_i \in \kvec} \dtheta[k_i]
    & \textrm{if }\omega = 0 \\
\Gkw[|\kvec|](\thetaw) \prod_{k_i \in \kvec} \dtheta[k_i]
    & \textrm{if }\omega = 1.\\
\end{cases}
\end{align*}
\end{defn}
%
The set $\kvec$ describes the ways and number of times $G(\thetaw, \wt)$ has
been differentiated through the first argument $\thetaw$, and $\omega$ indicates
whether or not a directional derivative of $G(\thetaw, \wt)$ with respect to the
second argument $\wt$ has been taken. The final argument of $\dterm$ denotes the
weight at which the derivative is evaluated. Note that the value of $\dterm$ is
invariant to the order of the elements of $\kvec$, so we can effectively treat
$\kvec$ as an unordered set. Using \defref{deriv_terms}, we can rewrite
\eqref{dg1} and \eqref{dg2} as follows:
\begin{align}\eqlabel{dgw_term_form}
\derivat{G(\thetaw[w], w)}{w}{\wt} \deltaw &=
    \dterm[\{1\}][0][\wt] + \dterm[\emptyset][1][\wt] \\
\derivat{{}^2 G(\thetaw[w], w)}{w dw}{\wt} \deltaw \deltaw &=
    \dterm[\{2\}][0][\wt] + \dterm[\{1, 1\}][0][\wt] +
    2 \dterm[\{1\}][1][\wt]. \nonumber
\end{align}
Analogously, $\dtheta[1]$ and $\dtheta[2]$ can also be written in terms of
\defref{deriv_terms}.
\begin{align*}
\dtheta[1] &= -\Hw^{-1} \dterm[\emptyset][1][\wt]\\
\dtheta[2] &= -\Hw^{-1}\left(
    \dterm[\{1, 1\}][0][\wt] + 2 \dterm[\{1\}][1][\wt]\right).
\end{align*}
The advantages (which we prove below) of \defref{deriv_terms} are as follows.
\begin{enumerate}
\item Every higher-order derivative of $G(\wt)$
    with respect to $\wt$ can be expressed as a linear combination of
    terms of the form $\dterm$.
\item Under the regularity assumptions given below, we can explicitly control the norm
    $\norm{\dterm}_{2}$.
\end{enumerate}
Consequently, every directional derivative $\dtheta$ can expressed and
controlled using terms of the form $\dterm$, which in turn controls the error of
the residual in the Taylor series $\thetaij[k-1]$.  The next section formalizes
this intuition.

\section{Assumptions and Formal Results}\seclabel{results}
\subsection{Assumptions}\seclabel{assumptions}

We will state our assumptions and results for a fixed $\kij \ge 0$ and a single,
fixed $w$ at which we wish to consider $\thetaij[\kij][w]$.  In practice, we
will typically require these assumptions to hold for a set of weights, e.g., all
leave-$\kappa$-out weights, or for some fixed number of random bootstrap
weights.  Such results can be obtained by requiring or demonstrating that the
assumptions laid out below hold uniformly over the relevant set of weights, as
we elaborate on in \secref{complexity_assumptions}.

We will need to consider all weights in an open neighborhood of a line
connecting $\onevec$ and $w$.
%
%
\begin{defn}\deflabel{wdom}
%
For a fixed $w$ and some arbitrarily small scalar $\wepsilon > 0$, define
\begin{align*}
\wdom := \left\{ \textrm{All } \wt \textrm{ such that }
    \norm{\wt - \left((1 - t) \onevec + t w \right)}_2 <
    \wepsilon \textrm{ for }t \in [0, 1] \right\}.
\end{align*}
\end{defn}
%
%
To disambiguate from the fixed $w$, we will denote members of $\wdom$ by $\wt
\in \wdom$.  Note that \eqref{g_definition} may, in general, be satisfied by
multiple $\theta$ for any particular $\wt \in \wdom$.  For example, in
optimization problems, these multiple solutions would correspond to multiple
local optima or saddle points. Despite this ambiguity, for the moment let
$\thetaw$ denote \textit{any} solution of \eqref{g_definition} for $\wt \in
\wdom$.  We will show in \lemref{implicit_function_thm} that, under the
assumptions we will now articulate, $\thetaw$ is in fact uniquely defined.

We denote the domain of $\theta$ as $\allthetadom$. For the fixed $w$ and some
$\epsilon>0$, we will require $\allthetadom$ to contain an open neighborhood of
the convex hull of all possible solutions $\thetaw[\wt]$ for $\wt \in \wdom$.
Let $\mathrm{Conv}(\{\cdot\})$ denote the convex hull of its argument.
%
%
\begin{assu}\assulabel{thetadom} 
Assume that, for each $\wt \in \wdom$, there exists at least one
$\thetaw$ such that $G(\thetaw, \wt) = 0$. Denote this set of solutions
$\left\{ \thetaw: \wt \in \wdom \right\}$.  Fix some scalar $\epsilon > 0$.
Assume that $\allthetadom$ is an open set containing an $\epsilon$-enlargement
of the convex hull of the set of all such $\thetaw$, i.e.,
\begin{align*}
\allthetadom \supseteq
\left\{\textrm{All }\theta \textrm{ such that }
    \norm{\theta - \tilde{\theta}}_2 < \epsilon
    \textrm{ for some } \tilde{\theta} \in
        \mathrm{Conv}\left(\left\{ \thetaw: \wt \in \wdom \right\}\right)
 \right\}.
\end{align*}
\end{assu}
%
We emphasize again that, under \assuref{thetadom}, $\allthetadom$ need not
contain every $\theta$ that satisfies $G(\theta, \wt)=0$ for some $\wt \in
\wdom$. All that is required is that, for every $\wt \in \wdom$, there is at
least one solution to $G(\theta, \wt)=0$ in $\allthetadom$.  We will
show later in \lemref{implicit_function_thm} that our assumptions are sufficient
to establish that $\thetaw$ is in fact unique within $\allthetadom$.

A domain for $\theta$ that satisfies \assuref{thetadom} depends on $w$.  Because
our error bounds will be in terms of suprema over $\theta \in \allthetadom$, it
will be advantageous to choose $\allthetadom$ to be as small as possible while
still satisfying \assuref{thetadom}.


\begin{figure}
    \centering
\includegraphics[width=0.9\textwidth]{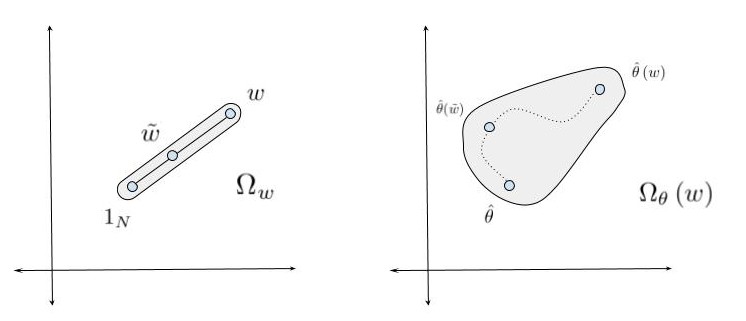}
\caption{A visualization of $\wdom$ and $\allthetadom$ for fixed $w$.}
\end{figure}

We now state mild assumptions on the objective evaluated at the original
weight vector $\onevec$, $G(\theta, \onevec)$.

\begin{assu}\assulabel{diffable} 
For all $\theta \in \allthetadom$, and all $n=0,\ldots,N$, $g_n(\theta)$ is
$\kij + 1$ times continuously differentiable.
\end{assu}

\Assuref{diffable} amounts to assuming that $\Gk(\theta, \wt)$ exists for all $k
\le \kij + 1$, all $\theta \in \allthetadom$, and all $\wt \in \wdom$.

\begin{assu}\assulabel{convexity} 
For all $\theta \in \allthetadom$, $\Hw[\theta, \onevec]$ is strongly positive
definite.  Specifically, there exists a finite constant $\cop$ such that
\begin{align*}
\sup_{\theta \in \allthetadom}
    \norm{\Hw[\theta, \onevec]^{-1}}_{op} \le \cop.
\end{align*}
\end{assu}
%
If $G(\theta, \onevec)$ arises as the first-order condition for a minimization
problem, \assuref{convexity} amounts to assuming that the objective function is
strongly convex, at least in a neighborhood of the local optimum.

\begin{assu}\assulabel{bounded} 
For all $\theta \in \allthetadom$, and each $1 \le k \le \kij + 1$, there exist
finite constants $\mk$ such that
\begin{align*}
    \sup_{\theta \in \allthetadom} \norm{\Gk(\theta, \onevec)}_{2} \le \mk.
\end{align*}
\end{assu}
%
$\mk$ is denoted with an ``M'' because it bounds the $L_2$ norms of the sample
means $\frac{1}{N} \sum_{n=0}^N \gnk(\theta)$. Note that $\Gk(\theta, \onevec)$
contains $D^{k + 1}$ entries, and so, in general, we expect $\mk$ to be of order
$D^{\frac{k + 1}{2}}$.\footnote{This fact leads to poor dimension dependence in
our bounds, which can be traced to our repeated use of the Cauchy-Schwartz bound
in the proof of \lemref{taylor_theta_term_bounds} below. Refining
\lemref{taylor_theta_term_bounds} in more structured cases to achieve better
dimension dependence is an interesting avenue for future work.}

\gonlyassums{} state requirements on the objective function with the unit weights
$\onevec$.  We now turn to conditions on the weight vector, $w$.
%
\begin{assu}\assulabel{complexity} 
For every $0 \le k \le \kij + 1$,
\begin{align*}
\sup_{\wt \in \wdom} \sup_{\theta \in \allthetadom}
    \norm{\Gk(\theta, \wt) - \Gk(\theta, \onevec)}_2
    \le \deltak.
\end{align*}
\end{assu}
%
Sometimes it will be useful to keep track only of the dependence on an upper
bound of all the set complexities, so we make the following definition.
\begin{defn}\deflabel{deltamax}
Define
$\deltamax := \max\left\{ \deltak[0], \ldots, \deltak[\kij + 1]\right\}$.
\end{defn}
\assuref{complexity} is a measure of the set complexity of the derivatives of
$G(\theta, \onevec)$ \citep{giordano:2019:ij}.\footnote{The reader may wonder
why we do not define \assuref{complexity} more compactly in terms of
$\Gkw(\theta)$. The reason is that the identity between $\Gk(\theta, w) -
\Gk(\theta, \onevec)$ and $\Gkw(\theta)$ holds only when $G(\theta, w)$ is
linear in $w$. \assuref{complexity} is stated in a form designed to illuminate
to its role in the proofs, which essentially depend on the smoothness of
$\Gk(\theta, w)$ in $w$ rather than properties the partial derivative of
$\Gk(\theta, w)$ with respect to $w$ evaluated at $\onevec$.}
Note that \assurangeref{thetadom}{bounded} depend on the ``regularization term''
$\gnk[][0](\theta)$, but \assuref{complexity} does not because
$\gnk[][0](\theta)$ does not depend on $w$ and so cancels in the difference.


When considering the behavior of $\thetaij$ for large $N$, our theory is
motivated by the presumption that $\cop$ and $\mk$ will remain finite but
nonzero, and that $\deltak$ will approach zero as $N$ gets large. We justify
this intuition and discuss the control of $\deltak$ in detail in
\secref{complexity_assumptions}.

\Assuref{convexity} requires $\Hone$ to be strongly positive definite at the
vector of weights $\onevec$, but we additionally require $\Hw$ to be
strongly positive definite for all $\wt \in \wdom$. To this end we state a
final, more restrictive condition that combines the above assumptions and
\assuref{complexity}.
%
\begin{cond}\condlabel{theta_small} 
Define the constant $\cset := \cop \deltak[1] + \cop^2 \lh \deltak[0]$.
Fix some $0 < \deltabound < 1$, and assume that $\cset \le \deltabound.$
\end{cond}
%
\Condref{theta_small} is not intuitive; it is simply a technical condition that
gives us control over the smallest eigenvalues of $\Hw$, as we show in
\lemref{weighted_hessian_inverse} below. Note that, to fulfill
\condref{theta_small}, it suffices to have $\deltamax \le \deltabound \left(\cop +
\lh\cop^2 \right)^{-1}$.~\footnote{We can compare \condref{theta_small} with the
related bound from \citet{giordano:2019:ij}. Both scale by $\cop^{-2}$ and
$\lh^{-1}$.  The dimension dependence $D$ has been removed by using the $L_2$
norm in \lemref{operator_norm_continuity}.}
By design, \condref{theta_small} gives the following lemma.
%
\begin{lem}\lemlabel{weighted_hessian_inverse}
\seeproof{weighted_hessian_inverse}
Define $\copw := \frac{1}{1-\deltabound} \cop$.  Under
\allassums{},
\begin{align*}
    \sup_{\wt \in \wdom} \norm{\Hw[\wt]^{-1}}_{op}  \le \copw.
\end{align*}
\end{lem}
%
As $\deltabound \rightarrow 1$, \lemref{weighted_hessian_inverse} becomes
vacuous, and as $\deltabound \rightarrow 0$, $\copw \rightarrow \cop$. We do not
specify the value of $\deltabound$, but, for simplicity, we take $\deltabound$
to be fixed.\footnote{Note that \citet{giordano:2019:ij} took $\deltabound =
\frac{1}{2}$.}
A key consequence of \lemref{weighted_hessian_inverse} is the existence of
$\thetaw \in \allthetadom$ as a smooth function of $\wt \in \wdom$.

\begin{lem}\lemlabel{implicit_function_thm}
\seeproof{implicit_function_thm}
Under \allassums{}, for any $\wt \in \wdom$, any solution $\thetaw \in
\allthetadom$ of $G(\theta, \wt) = 0$ is a  $\kij + 1$ times continuously
differentiable function of $\wt$.
\end{lem}

\Lemref{implicit_function_thm} justifies our use of the chain rule in
\secref{taylor_expansion} above, and, indeed, our use of the notation
$\thetaw$.

\subsection{Controlling the error}\label{sec:error}
We now put the above pieces together to control the error $\errij[\kij]$ of the
Taylor series approximation $\thetaij[\kij]$.
We begin by proving the two assertions made at the end of
\secref{taylor_expansion} about the utility of $\dterm$ in \defref{deriv_terms}.

\begin{lem}\lemlabel{higher_order_general_terms}
\seeproof{higher_order_general_terms}
Let \allassums{} hold.  Let $k$ be any integer in $1,\ldots,\kij + 1$.
Then there exists a set of tuples $\dtermargs[k] = \left\{ (a_i, \kvec_i,
\omega_i), i \in 1,\ldots, |\dtermargs[k]| \right\}$, where $a_{i}$ are fixed
scalars and $\kvec_i$ and $\omega_{i}$ are defined as in \defref{deriv_terms},
such that
\begin{align*}
\dtheta
    =\Hw^{-1}\left(
        \sum_{(a_i, \kvec_i, \omega_i) \in \dtermargs[k]}
        a_i \dterm[\kvec_i][\omega_i][\wt] \right).
\end{align*}
Furthermore, each of the tuples $(a_i, \kvec_i, \omega_i)$ satisfies the following
properties:
\begin{align}
\max_{k' \in \kvec_{i}}k' & < k \quad\textrm{and} \eqlabel{dtheta_term_properties} \\
\omega_i + \sum_{k'\in\kvec_{i}}k' &= k. \eqlabel{dtheta_term_total_order}
\end{align}
\end{lem}

\Eqref{dtheta_term_properties} of \lemref{higher_order_general_terms} states
that $\dtheta[k]$ can be computed using only terms involving $\dtheta[k']$ for
$k' < k$, and \Eqref{dtheta_term_total_order} states that each such term arises
from having taken $k$ directional derivatives of $G(\thetaw, \wt)$.
\lemref{higher_order_general_terms} thus formalizes the idea that the process
used in \secref{taylor_expansion} to compute the first two derivatives
$\dtheta[1]$ and $\dtheta[2]$ can be continued to arbitrarily high orders.
Additionally, \Lemref{higher_order_general_terms} asserts the computationally
useful fact that $\dtheta[][\onevec]$ of all orders can be computed with a
single matrix inverse, $\Hone^{-1}$.

\begin{lem}\lemlabel{taylor_theta_term_bounds}
\seeproof{taylor_theta_term_bounds}
Under \allassums{}, for all $\wt \in \wdom$,
\begin{align*}
\norm{\dterm[\kvec][\omega][\wt]}_2
    \le (\deltak[|\kvec|] + (1 - \omega) \mk[|\kvec|])
        \prod_{k \in \kvec} \norm{\dtheta[k]}_2.
\end{align*}
If $\kvec = \emptyset$, the empty product is taken to be $1$.
\end{lem}

With these results in hand, we can start to build the inductive ladder that
will bound the error of $\thetaij$ for all orders up to $\kij - 1$.
Our main result will proceed by induction on the order of derivatives
of $d^{k}\hat{\theta}\left(\wt\right)$. We now take the first step
in the induction.
%
%
\begin{cor}\corlabel{dtheta_bounded}
%
Under \allassums{}, 
\begin{align*}
\sup_{\wt \in \wdom} \norm{\dtheta[1]}_{2} & \le \copw \deltak[0].
\end{align*}
\begin{proof}\prooflabel{dtheta_bounded}
By \eqref{dtheta1}, $\dtheta[1] = -\Hw^{-1} \dterm[\emptyset][1][\wt]$, so
\begin{align*}
\norm{\dtheta[1]}_{2} & \le \copw
    \norm{\dterm[\emptyset][1][\wt]}_{2}
    &\textrm{(\lemref{weighted_hessian_inverse})} \\
& \le \copw \deltak[0].
    &\textrm{(\lemref{taylor_theta_term_bounds})}.
\end{align*}
\end{proof}
\end{cor}
%
Note that we were able to apply \lemref{taylor_theta_term_bounds} with $\kvec =
\emptyset$, and so did not require bounds on any other derivatives to begin the
induction.

Recall that we expect our theory to apply most usefully in situations where
$\deltamax \rightarrow 0$ as $N \rightarrow \infty$ (see
\secref{complexity_assumptions} below where we establish that $\deltamax
\rightarrow 0$ under common stochastic assumptions).
As a preliminary result, we observe that the assumptions of \secref{assumptions}
and smallness of $\deltak[0]$ suffice to prove that $\thetaw$ is close to
$\thetaone$.
%

\begin{cor}\corlabel{theta_difference_bound}
\seeproof{theta_difference_bound}
Under \assuref{thetadom}, \assuref{diffable}, \assuref{convexity}, and
\assuref{complexity},
\begin{align*}
\sup_{\wt \in \wdom} \norm{\thetaw - \thetaone}_{2} & \le \cop \deltak[0].
\end{align*}
\end{cor}
%

When $\deltamax
\rightarrow 0$, \corref{dtheta_bounded} and \corref{theta_difference_bound}
trivially prove that the difference between $\thetaij[1]$ and
$\thetaw[w]$ goes to zero because both go to $\thetaone$.
\begin{align*}
    \norm{\thetaij[1] - \thetaw[w]}_2 &\le
        \norm{\thetaij[1] - \thetaone}_2 + \norm{\thetaw[w] - \thetaone}_2  \\
    &\le \norm{\dtheta[1]}_2 + \norm{\thetaw[w] - \thetaone}_2 \\
    &\le (\cop + \copw) \deltamax.
\end{align*}
Of course, an ``approximation'' to $\thetaw[w]$ that is identically $\thetaone$
will share this result---indeed, with a better error bound of
$\norm{\thetaw[w] - \thetaone}_2 \le \cop \deltamax$. The rate given by this
trivial reasoning is proportional to $\deltamax$.  We now show that a more
careful analysis provides a faster rate in the form of bounds that are
higher powers of $\deltamax$.

\begin{cor}\corlabel{d2theta_bounded}
%
Under \allassums{},
\begin{align*}
\sup_{\wt \in \wdom} \norm{\dtheta[2]}_{2}  \le
\copw^3  \mk[2] \deltak[0]^2 +
    2 \copw^2 \deltak[1] \deltak[0] +
    \copw^3 \deltak[2] \deltak[0]^2.
\end{align*}
%
\begin{proof}\prooflabel{d2theta_bounded}
As shown in \secref{taylor_expansion}, $\dtheta[2] = -\Hw^{-1}\left(
\dterm[\{1, 1\}][0][\wt] + 2 \dterm[\{1\}][1][\wt]\right)$, so
\begin{align*}
\norm{\dtheta[2]}_{2} & \le \copw
    \norm{\dterm[\{1, 1\}][0][\wt] +
          2 \dterm[\{1\}][1][\wt]}_{2}
    &\textrm{(\lemref{weighted_hessian_inverse})} \\
& \le \copw \left(
    \norm{\dterm[\{1, 1\}][0][\wt]}_2 +
    2 \norm{\dterm[\{1\}][1][\wt]}_2 \right)
     & \textrm{(triangle inequality)} \\
& \le \copw \left(
    (\deltak[2] + \mk[2]) \norm{\dtheta[1]}_2^2 +
    2 \deltak[1] \norm{\dtheta[1]}_2
    \right)
    & \textrm{(\lemref{taylor_theta_term_bounds})} \\
& \le \copw \left(
    (\deltak[2] + \mk[2]) (\copw^2 \deltak[0]^2) +
    2 \deltak[1] (\copw \deltak[0])
    \right) & \textrm{(\corref{dtheta_bounded})}\\
& = \copw^3  \mk[2] \deltak[0]^2 +
    2\copw^2 \deltak[1] \deltak[0] +
    \copw^3 \deltak[2] \deltak[0]^2. & 
\end{align*}
\end{proof}
\end{cor}
%
For evaluating actual error bounds in practice, it can be helpful to keep track
of the dependence of bounds on $\deltak$ for different $k$ as we have done in
\corref{d2theta_bounded}.  However, for asymptotics and theoretical intuition,
it will suffice to simplify expressions to depend only on $\deltamax$, which we
will do for the remainder of this section.

An immediate consequence of \corref{d2theta_bounded} is
%
\begin{cor}\corlabel{thetaij1_consistent}
%
Applying \corref{d2theta_bounded} and \lemref{taylor_series}, we have
\begin{align*}
\norm{\thetaij[1] - \thetaw[w]}_2 \le
\left(\copw^3  \mk[2] + 2 \copw^2 \right) \deltamax^2 + \copw^3 \deltamax^3.
\end{align*}
\end{cor}

We can see that \corref{thetaij1_consistent} provides qualitatively similar
results as \citet{giordano:2019:ij}.  In particular, we see that the leading
dependence of the error is $O(\deltamax^2)$, $O(\cop^3)$, and $O(\mk[2])$.

In the proof of \corref{d2theta_bounded}, we used the bound on
$\norm{\dtheta[1]}_2$ provided by \corref{dtheta_bounded} to control the size of
$\norm{\dtheta[2]}_w$. \Eqref{dtheta_term_properties} of
\lemref{higher_order_general_terms} guarantees that this inductive process can
be continued, using control of $\norm{\dtheta[1]}_2$ and $\norm{\dtheta[2]}_2$
to control $\norm{\dtheta[3]}_2$, and so on.
The expressions for $\errij(w)$ are defined inductively and will grow
complicated for larger $\kij$, but their calculation is entirely mechanical.
Given values for $\copw$, $\mk$, and $\deltak$, one can easily evaluate the
error bounds for any order numerically.

However, without keeping explicit track of constants, certain aspects of the
bounds can be illuminating.  For example, the inductive procedure immediately
gives rise to the following theorem, which relates the asymptotic accuracy of
$\thetaij[\kij]$ to the set complexity $\deltamax$.

\begin{thm}\thmlabel{thetaij_consistent}
\seeproof{thetaij_consistent}
Assume that \allassums{} hold for  for all $k=0,\ldots,\kij$ for fixed $\copw$ and
$\mk$ as $\deltamax \rightarrow 0$.  Then, for all $k=0,\ldots,\kij$, as $\delta
\rightarrow 0$,
\begin{align*}
    \sup_{\wt \in \wdom} \norm{\dtheta[k + 1]}_2 = O\left(\deltamax^{k + 1}\right)
    \quad\quad\textrm{and}\quad\quad
    \norm{\thetaij - \thetaw[w]}_2 = O\left( \deltamax^{k + 1} \right).
\end{align*}
\end{thm}

%
%

\subsection{Bounds on complexity}\label{sec:complexity_assumptions}
In this section, we provide bounds on $\deltak$ of \assuref{complexity} for CV
weights under various stochastic assumptions.  In particular, we show that, in
many cases of interest, $\deltamax$ goes to zero as the number of data points $N$
goes to infinity.  Together with \thmref{thetaij_consistent} above, the results
of this section gives consistency results for CV under various assumptions.

We begin by observing that, although \assuref{complexity} is stated as a
supremum over all $\wt \in \wdom$, under \assuref{thetadom} and
\assuref{bounded} it suffices to control the set complexity only for $\wt = w$,
as we now show in the following lemma.
\begin{lem}\lemlabel{w_complexity}
Suppose that \assuref{thetadom} and \assuref{bounded} hold, and
that, for each $0 \le k \le \kij + 1$,
\begin{align*}
\sup_{\theta \in \allthetadom}
    \norm{\Gk(\theta, w) - \Gk(\theta, \onevec)}_2 \le \tilde{\delta}_k.
\end{align*}
Let $\wepsilon$ be the (arbitrarily small) enlargement constant given in
\defref{wdom}. Then, for each $0\le k \le \kij + 1$, \assuref{complexity} is
satisfied with
\begin{align*}
    \deltak = \wepsilon \mk + \tilde{\delta}_k. \\
\end{align*}
\end{lem}
\begin{proof}
By \defref{wdom}, for every $\wt \in \wdom$, there exists a $t \in [0, 1]$ such
that $\norm{\wt - (\onevec + t(w - \onevec))}_2 < \wepsilon$, and so
\begin{align*}
\MoveEqLeft
\sup_{\wt \in \wdom} \sup_{\theta \in \allthetadom}
    \norm{\Gk(\theta, \wt) - \Gk(\theta, \onevec)}_2 \\
= & \sup_{\theta \in \allthetadom}
    \wepsilon \norm{\Gk(\theta, \onevec)}_2 + {} \\
  & \sup_{\theta \in \allthetadom}
        \norm{\Gk(\theta, (\onevec + t(w - \onevec)) -
            \Gk(\theta, \onevec)}_2 \\
\le &
    \wepsilon \mk +
    t \sup_{\theta \in \allthetadom}
        \norm{\Gk(\theta, w) - \Gk(\theta, \onevec)}_2 \\
\le & \wepsilon \mk + t \tilde{\delta}_k \\
\le & \wepsilon \mk + \tilde{\delta}_k.
\end{align*}
\end{proof}

In \defref{wdom}, we require that $\wepsilon > 0$ only to satisfy a technical
condition in the proof of \lemref{implicit_function_thm} (specifically, that
$\wdom$ be an open set). We are free to take a different value of $\wepsilon$
for each $N$, say $\wepsilon_N$, and let $\wepsilon_N$ got to zero as $N
\rightarrow \infty$ at any rate we like.  Consequently, by
\lemref{w_complexity}, it suffices to control $\sup_{\theta \in \allthetadom}
\norm{\Gk(\theta, w) - \Gk(\theta, \onevec)}_2$, and a corresponding bound will
apply for all $\wt \in \wdom$.

For the remainder of this section, we denote $[N] := \{1, \ldots, N\}$ to be the
set of observation indices, and let $\nset \subseteq [N]$ be some subset of
indices.

\begin{assu}\assulabel{bounded_var} 
For all $\theta \in \allthetadom$ and each $0 \le k \le \kij + 1$, there exist
constants $\vk$ such that
\begin{align*}
\sup_{\theta \in \allthetadom}
    \frac{1}{N} \sum_{n=1}^N \norm{\gnk(\theta)}_{2}^2 \le \vk.
\end{align*}
\end{assu}

\begin{assu}\assulabel{bounded_terms} 
Fix a set of $\nset$ of indices.
Then, for all $\theta \in \allthetadom$, each $n \in \nset$, and each $0 \le
k \le \kij + 1$, there exist constants $\tk$ such that
\begin{align*}
\max_{n \in \nset } \sup_{\theta \in \allthetadom} \norm{\gnk(\theta)}_{\infty}
    \le \tk.
\end{align*}

\end{assu}

Note that \assuref{bounded_terms} implies \assuref{bounded_var} with $\vk =
D^k \tk$.  In a sense, \assuref{bounded_var, bounded_terms} are stronger
than \assuref{bounded} although, unlike \assuref{bounded},
\assuref{bounded_var, bounded_terms} do not imply any bounds on the
regularization term $\gnk[][0](\theta)$. However,
\assuref{bounded} holds under \assuref{bounded_var} with $\mk = \sqrt{\vk} +
\sup_{\theta \in \allthetadom} \norm{\gnk[][0](\theta)}_2$ by
\propref{array_norm_relation}.

\begin{lem}\lemlabel{cv_complexity}
%
Fix a set of $\nset \subseteq \{1, \ldots, N \}$ of indices.  Let
$w_{(-n)}$ denote a weight vector that is $0$ in its $n$-th index
and $1$ otherwise.  Define a set of leave-one-out CV weights
\begin{align*}
\loodom := \left\{ w: w = w_{(-n)} \textrm{ for some }n \in \nset\right\}.
\end{align*}
Then, for all $w \in \loodom$,
\begin{align*}
\mathrm{\assuref{bounded_var}} &\quad\Rightarrow\quad
\textrm{\assuref{complexity} holds with }
    \deltak = \frac{\sqrt{\vk}}{\sqrt{N}}
    \quad\quad{ and } \\
\mathrm{\assuref{bounded_terms}} &\quad\Rightarrow\quad
    \textrm{\assuref{complexity} holds with }
    \deltak = \frac{\tk}{N}.
\end{align*}


\begin{proof}\prooflabel{cv_complexity}
For the first result, we write
\begin{align}
\MoveEqLeft
\max_{w \in \loodom} \sup_{\theta \in \allthetadom}
    \norm{\Gk(\theta, w) - \Gk(\theta, \onevec)}_2 \nonumber\\
& = \max_{w \in \loodom} \sup_{\theta \in \allthetadom}
    \norm{ \frac{1}{N}\sum_{n=1}^N \gnk(\theta) \deltaw_n }_2 & \nonumber\\
& = \frac{1}{N} \max_{n \in \nset} \sup_{\theta \in \allthetadom} \norm{\gnk(\theta)}_2
    &\textrm{(def. of LOO CV weights)} \eqlabel{tk_term}\\
& \le \frac{1}{\sqrt{N}}
    \sqrt{ \sup_{\theta \in \allthetadom}
        \frac{1}{N} \sum_{n=1}^N \norm{\gnk(\theta)}_2^2}
    & \nonumber \\
& \le \frac{\sqrt{\vk}}{\sqrt{N}}.
    & \textrm{(\assuref{bounded_var})} \nonumber
\end{align}

The second result follows immediately from the previous display by bounding
\eqref{tk_term} by $\tk / N$ under \assuref{bounded_terms}.

\end{proof}
%
%
\end{lem}

In order to relate the above assumptions to real data, let us place two
additional potential conditions on $\gnk(\theta)$.

\begin{cond}\condlabel{iid} 
For $0 \le k \le \kij + 1$ and for all $n \ge 1$, the functions
$\gnk(\theta)$ are independent and identically distributed (IID) with
\begin{align*}
\mbe\left[ \norm{\gnk(\theta)}_2^2 \right] < \infty.
\end{align*}
\end{cond}

\Condref{iid} arises, for example, when minimizing an empirical loss function
with IID data $x_n$, where $g_n(\theta) := \partialat{f(x_n,
\theta)}{\theta}{\theta}$, and where the distribution of the data $x_n$ is
sufficiently light-tailed.

\begin{cond}\condlabel{supterms} 
For all $0 \le k \le \kij + 1$ and for all $n \ge 1$, assume that
\begin{align*}
\gsupnk :=  \sup_{\theta \in \allthetadom}\norm{\gnk(\theta)}_\infty < \infty.
\end{align*}
\end{cond}

\Condref{supterms} occurs when, for example, $\gnk(\theta)$ is continuous on the
closure of $\allthetadom$.  Even when \condref{supterms} holds
$\lim_{N\rightarrow\infty}\tk$ may be infinite, as \condref{supterms} controls
only the supremum over $\theta$ but not the maximum over $n$. However, as
\lemref{cv_complexity_sampled} shows, \condref{supterms} does imply that $\tk$
will remain bounded for all $N$ with high probability when $\nset$ contains a
fixed number of randomly selected set of indices.

\begin{cond}\condlabel{supterms_bounded} 
Let \condref{supterms} hold.  For all $0 \le k \le \kij + 1$ assume that
\begin{align*}
\sup_{n \ge 1}\gsupnk < \infty.
\end{align*}
\end{cond}

\Condref{supterms_bounded} occurs, for example, when both $\theta$ and the data
$x_n$ are constrained to compact sets and, additionally, $\gnk(\theta)  =
\partialat{{}^{k + 1}f(\theta, x_n)}{\theta^{k+1}}{\theta}$ are continuous in
both $\theta$ and $x_n$ for all $n$ and $k$. \Condref{supterms_bounded} is a
restrictive assumption, but it is one that is commonly made in the IJ literature
(see \secref{lit_review}).


\begin{lem}\lemlabel{cv_complexity_sampled}
%
Let \condref{iid, supterms} hold, and let $\nset$ contain a fixed number,
$N_{loo}$, of indices chosen randomly from $n \in [N]$ according to any
distribution that is independent of the values $g_n(\theta)$, $0 \le n \le N$.
Pick any $\alpha \in (0, 1)$. Then, for each $0 \le k \le \kij + 1$, there
exists a constant $\quant$ (depending on $\alpha$ and $k$ but not depending on
$N$), such that, for all $N$,
\begin{align*}
P\left(
    \textrm{\assuref{bounded_terms} is satisfied with }\tk=\quant
\right) \ge 1 - N_{loo}(1 - \alpha),
\end{align*}
where $P$ denotes probability in both the random indices $\nset$ and the random
functions $g_n(\theta)$.
\begin{proof}
For $\alpha \in (0, 1)$,
let $\quant$ denote the $\alpha$-th quantile of $\gsupnk$.  By \condref{iid},
$\quant$ is the same for all $n$. Then, for each $0 \le k \le \kij + 1$, and for
all $N$.  By definition, $P\left(\gsupnk > \quant \right) \le (1 - \alpha)$.  So
\begin{align*}
P\left(\max_{n\in\nset} \gsupnk \le \quant\right)
    &= 1 - P\left(\bigcup_{n\in\nset} \gsupnk > \quant\right) \\
    &\ge 1 - \sum_{n \in \nset} P\left(\gsupnk > \quant \right)
        &\textrm{(union bound)}\\
    &\ge 1 - N_{loo} (1 - \alpha).
\end{align*}
\end{proof}
\end{lem}


Using the above conditions, we can now describe different asymptotic regimes for
the accuracy of the HOIJ when approximating leave-one-out CV.  The following
results will require that \gonlyassums{} hold ``uniformly for all $N$'', by
which we mean that \gonlyassums{} hold for all $N$ with constants that do not
depend on $N$.  \proprangeref{cv_complexity_var}{cv_complexity_uniform} make
probability statements that hold in the random functions $g_n(\theta)$
and, when applicable, to the random indices in $\nset$.

\begin{prop}\proplabel{cv_complexity_var}
Fix $\kij$.  Let \condref{iid} hold, and let \gonlyassums{} hold for uniformly
for all $N$.  Let $\loodom$ denote the set of all leave-one-out CV weights.
Then, as $N \rightarrow \infty$,
\begin{align*}
\max_{w \in \loodom} \norm{\thetaij[\kij] - \thetaw[w]}_2 =
    O_p(N^{-\frac{1}{2}(\kij + 1)}).
\end{align*}
\begin{proof}\prooflabel{cv_complexity_var}
Under \condref{iid}, for each $0 \le k \le \kij + 1$, \assuref{bounded_var}
holds with $\vk = O_p(1)$ as $N \rightarrow \infty$ by the law of large numbers.
By \lemref{cv_complexity}, \assuref{complexity} holds with $\deltak =
O_p(N^{-1/2})$. Because $\deltak \rightarrow 0$ in probability as $N \rightarrow
\infty$, and because \gonlyassums{} hold uniformly for all $N$,
\begin{align*}
    \cop \deltak[1] + \cop^2 \lh \deltak[0] = o_p(1),
\end{align*}
and so $P(\textrm{\condref{theta_small} holds}) \rightarrow 1$. Consequently,
with probability approaching 1 we can apply \thmref{thetaij_consistent}, and the
result follows.
\end{proof}
\end{prop}



\begin{prop}\proplabel{cv_complexity_random}
Let \condref{iid, supterms} hold, and let \gonlyassums{} hold uniformly for all
$N$.  Let $\loodom$ denote a set $N_{loo}$ leave-one-out CV weights, randomly
chosen as in \lemref{cv_complexity_sampled}, where $N_{loo}$ remains fixed for
all $N$. Fix $\kij$.  Then, for any $\epsilon > 0$, with probability at least $1 -
\epsilon$,
\begin{align*}
\max_{w \in \loodom} \norm{\thetaij[\kij] - \thetaw[w]}_2 =
    O_p(N^{-(\kij + 1)}),
\end{align*}
\begin{proof}
Take $\alpha \ge 1 - N_{loo} \epsilon$ in \lemref{cv_complexity_sampled}, which
we can apply thanks to \condref{iid, supterms}. Then, with probability $1 -
\epsilon$, \assuref{bounded_terms} holds, and thus, by \lemref{cv_complexity},
\assuref{complexity} holds for each $0 \le k \le \kij + 1$ with $\deltak =
O(N^{-1})$.  The remainder of the proof is identical to that of
\propref{cv_complexity_var}.
\end{proof}
\end{prop}



\begin{prop}\proplabel{cv_complexity_uniform}
Let \condref{supterms_bounded} hold, and let \gonlyassums{} hold uniformly for
all $N$.  Let $\loodom$ denote all possible leave-one-out CV weights. Fix
$\kij$.  Then, with probability one,
\begin{align*}
\max_{w \in \loodom} \norm{\thetaij[\kij] - \thetaw[w]}_2 =
    O(N^{-(\kij + 1)}).
\end{align*}
\begin{proof}
For each $0 \le k \le \kij + 1$, \condref{supterms_bounded} implies that
\assuref{bounded_terms} holds with finite $\tk$ independent of $N$, and so
\assuref{complexity} holds with $\deltak = O(N^{-1})$ by \lemref{cv_complexity}.
The remainder of the proof is identical to that of \propref{cv_complexity_var}.
\end{proof}
\end{prop}



%
%
%

\subsection{A note on the bootstrap}\label{sec:bootstrap}

In \secref{complexity_assumptions} we limited our consideration to leave-one-out
cross validation weights.  It is natural to consider ``larger'' weight vectors,
such as those of the bootstrap \citep{efron:1994:bootstrap}.  In particular, by
drawing $w$ from a multinomial distribution with $N$ trials, $N$ outcomes, and
probabilities all equal to $N^{-1}$, $\thetaw[w]$ is a sample from the
bootstrap distribution of the estimator $\thetaone$.  For sufficiently
well-behaved parametric function classes, we may expect to be able to apply
\assuref{complexity} with $\deltak = O_p(N^{-1/2})$ for bootstrap weights
because, in that case, \assuref{complexity} reduces to a supremum over a sample
correlation with respect to the empirical distribution \citep{giordano:2019:ij}.
Formally controlling $\deltak$ in this way is a delicate matter that we leave
for future work.

However, we briefly observe that it is particularly important to use at least a
second-order approximation ($\kij \ge 2$) when $w$ is a bootstrap weight vector.
This is because, in that case, the first-order approximation $\thetaij[1][w]$ is
equivalent to an asymptotic normal approximation, as we show in the following
proposition.

\begin{prop}\proplabel{linear_bootstrap}
Equivalence of the linear bootstrap and the sandwich estimator.
\seeproof{linear_bootstrap}
Let $w \sim \mathrm{Multinomial}\left(N, N^{-1}\onevec\right)$ be bootstrap
weights, and let $\cov_b(\cdot)$ denote the covariance taken with respect to the
randomness in $w$ alone.  Let $\widehat{\cov}(g_n(\thetaone))$ denote the
sample covariance of the $N$ observations $g_n(\thetaone)$.  Then
\begin{align*}
\cov_b\left(\thetaij[1][w]\right) =
    \Hone^{-1} \widehat{\cov}(g_n(\thetaone)) \Hone^{-1}.
\end{align*}
\end{prop}
The covariance in \propref{linear_bootstrap} is simply the sample version of the
standard asymptotic ``sandwich covariance matrix'' for M-estimators (see, e.g.,
\citet{stefanski:2002:mestimation}). Consequently, covariance estimates based on
the linear approximation $\thetaij[1][w]$ do not improve on asymptotic normal
theory.  In order to gain additional inferential benefit from using the
bootstrap, it is necessary to use higher-order Taylor series approximations to
$\thetaw[w]$ with $\kij \ge 2$.

\section{Computation}\label{sec:computation}
In order to compute the $\dtheta[k][\onevec]$ term in $\thetaij[\kij]$, we
must perform three computational tasks. First, we must compute
$\dtheta[k'][\onevec]$ for $k' < k$.  Because $\dtheta[1][\onevec]$ has a closed
form, we can proceed inductively using \lemref{higher_order_general_terms}.
Second, we must solve a linear system involving $\Hone$.  Though it may be
expensive to compute and factorize $\Hone$, once $\Hone$ is computed and
factorized the marginal cost of repeatedly solving this linear system will be
low. Finally, we must evaluate the $k'$-th order directional derivatives of
$G(\theta, \onevec)$ for $k' \le k$.

Because the dimension of $\Gk(\theta, \onevec)$ is $D^{k + 1}$ in general,
storing the array $\Gk(\theta, \onevec)$ in memory may be prohibitively
expensive for even moderately-sized $D$ and $k$.  However, to compute
$\thetaw[w]$ for any particular $w$ requires only the directional derivatives of
the form $\Gk(\theta, \onevec) v_1 \ldots v_k$ for some vectors $v_{k'} \in
\mathbb{R}^D$.  As we now illustrate with an example, consideration of the
memory requirements of $\Gk(\theta, \onevec)$ motivates the use of forward-mode,
rather than reverse-mode, automatic differentiation to compute the HOIJ.  See
\citet{maclaurin:2015:autograd} for a recent review of automatic
differentiation.

Let us consider the following example.  Suppose we have data $x_n \in
\mathbb{R}^D$ for $1\le n \le N$ and consider the estimating equation
\begin{align*}
G(\theta, w) =& \sum_{n=1}^{N} w_n \exp\left(\theta^T x_{n}\right)x_{n}.
\end{align*}
In the notation of \secref{notation}, such an estimating equation would follow
from minimizing the loss function arising from $f(\theta, x_n) = \exp(\theta^T
x_n)$ with $f_0(\theta) \equiv 0$.  Let $\otimes^{k}$ denote an outer product in
the sense that $\otimes^{2} x_n$ is the bilinear operator from $\mathbb{R}^D
\times \mathbb{R}^D$ to $\mathbb{R}$ given by $\left(\otimes^{2} x_n\right)(v_1,
v_2) = (x_n^T v_1) (x_n^T v_2)$, with an analogous definition for higher values
of $k$.  Then $\Gk(\theta, \onevec)$ and its directional derivative are given by
\begin{align}
\Gk(\theta, \onevec) =&
    \sum_{n=1}^{N} \exp\left(\theta^T x_{n}\right) \otimes^{k+1} x_n
    \nonumber \\
\Gk(\theta, \onevec) v_1 \ldots v_k =&
    \sum_{n=1}^{N} \exp\left(\theta^T x_{n}\right)
        \sum_{k'=1}^k x_n^T v_{k'}. \eqlabel{gk_example_directional}
\end{align}
The size of the array $\Gk(\theta, \onevec)$ is determined by the size of
$\otimes^{k+1} x_{n}$, which contains $D^{k+1}$ elements.  However, to compute
the directional derivative of \eqref{gk_example_directional} does not require
creating the matrix $\sum_{n=1}^N \exp\left(\theta^T x_{n}\right) \otimes^k x_n$
in memory, as one can calculate \eqref{gk_example_directional} in a single loop
over $N$ with memory requirements of only $O(D)$.

Recall that reverse-mode automatic differentiation is designed to return the
array $\Gk(\theta, \onevec)$, which then must be stored in memory for
application to $v_1,\ldots,v_k$. In contrast, forward-mode automatic
differentiation is designed to directly calculate products such as $\Gk(\theta,
\onevec) v_1,\ldots,v_k$, ideally without instantiating large derivative arrays
in the process. The main cost of using forward-mode automatic differentiation in
our context is that the derivatives $\Gk(\theta, \onevec)$ must be calculated
anew for every different set of directions $v_1,\ldots,v_k$, effectively paying
a computational cost in exchange for memory savings. When $D$ and $k$ are large,
this trade-off can be favorable for forward-mode automatic differentiation.

We now return to full generality and describe an algorithm to compute
$\thetaij[w]$ using forward-mode automatic differentiation. Formally,
forward-mode automatic differentiation maps a function, a function argument, and
a vector to the directional derivative of the function in the direction of the
vector evaluated at the argument, as represented in \algrref{fwd}.
\begin{algorithm}[H]
\caption{Forward-mode automatic differentiation}
\algrlabel{fwd}
\begin{algorithmic}
\Procedure{ForwardModeAD}
    {$f(\cdot):\mathbb{R}^D \rightarrow \mathbb{R}^{D_f}$,
     $\theta_0 \in \mathbb{R}^D$, $v\in\mathbb{R}^D$}
    \State \Return  $\partialat{f(\theta)}{\theta}{\theta_0} v$
\EndProcedure
\end{algorithmic}
\end{algorithm}
In general, the dimension $D_f$ of the domain of $f(\cdot)$ may be different
than the dimension of the range, though in our case we will always have $D_f =
D$.

By viewing $\fwd(f, \theta, v_0)$ itself as a function of $\theta$ for fixed
$f$ and $v_0$, we can again apply $\fwd$ to compute the second-order directional
derivative:
\begin{align*}
\fwd(\fwd(f, \cdot, v_0), \theta_0, v_1) =
    \partialat{{}^2 f(\theta)}{\theta \partial\theta}{\theta_0} v_0 v_1.
\end{align*}
By recursively applying $\fwd$ in this way, we can compute higher-order
directional derivatives. In general, the cost of each pass of $\fwd$ requires a
constant multiple of the evaluation time of $f(\theta)$ itself, and $\fwd$ can
often be implemented without instantiating higher-order arrays of derivatives in
memory.

We will now use $\fwd$ to articulate a recursive algorithm to compute
$\thetaij[\kij]$. As in \lemref{higher_order_general_terms}, let $\dtermargs[k]$
denote the tuples $(a_i, \kvec_i, \omega_i)$ associated with terms in $\dtheta$,
and let $\dtermargs[*] = \{\dtermargs[k] \textrm{ for }
k\in\{1,\ldots,\kij\}\}$.  Note that $\dtermargs[*]$ is the same for all
problems can be computed automatically in advance. Let $\dthetaset :=
\{\dtheta[k'][\onevec]: 0 \le k' \le k \}$ denote the set of directional
derivatives of $\thetaw$ of order less than or equal to $k$.

First, \algrref{dterm} evaluates a single derivative term (\defref{deriv_terms})
using the directional derivatives stored in $\dthetaset$ of the appropriate
order.  Next, \algrref{dtheta} uses \algrref{dterm} to evaluate $\dtheta$ using
derivatives of lower orders.  Finally, \algrref{thetaij} evaluates the Taylor
series using \algrref{dtheta}.  Ultimately, it is
\lemref{higher_order_general_terms} that guarantees that this recursive approach
is possible.
\begin{algorithm}[H]
\caption{Evaluate $\dterm[][][w]$}
\algrlabel{dterm}
\begin{algorithmic}
\Procedure{EvaluateTerm}
    {$\kvec$, $\omega$, $\dthetaset[\max_{k\in\kvec} k]$, $\thetaone$}
    \State $f(\cdot) \gets G(\cdot, \omega)$
    \For{$k \in \kvec$}
        \State $f(\cdot) \gets \fwd(f(\cdot), \cdot, \dtheta[k][\onevec])$
    \EndFor
    \State \Return $f(\thetaone)$
\EndProcedure
\end{algorithmic}
\end{algorithm}
\begin{algorithm}[H]
\caption{Evaluate $\dtheta[k][w]$}
\algrlabel{dtheta}
\begin{algorithmic}
\Procedure{EvaluateDTheta}
    {$k$, $\thetaone$, $\Hone^{-1}$, $\dtermargs[k]$, $\dthetaset[k-1]$}

    \State $d \gets 0$
    \For{$(a_i, \kvec_i, \omega_i) \in \dtermargs[k]$}
        \State $d \gets d + a_i \cdot $
            \Call{EvaluateTerm}
                {$\kvec_i$, $\omega_i$, $\dthetaset[k-1]$, $\thetaone$}
    \EndFor
    \State \Return $-\Hone^{-1} d$
\EndProcedure
\end{algorithmic}
\end{algorithm}
\begin{algorithm}[H]
\caption{Evaluate $\thetaij[\kij][w]$}
\algrlabel{thetaij}
\begin{algorithmic}
\Procedure{EvaluateThetaIJ}
    {$\kij$, $\thetaone$, $\Hone^{-1}$, $\dtermargs[*]$}
    \State $\dthetaset[0] \gets \emptyset$
    \State $t \gets \thetaone$
    \For{$k \in \{1,\ldots,\kij\}$}
        \State $\dtheta[k][w] \gets $ \Call{EvaluateDTheta}
            {$k$, $\thetaone$, $\Hone^{-1}$, $\dtermargs[k]$, $\dthetaset[k-1]$}
        \State $\dthetaset[k] \gets \dthetaset[k-1] \bigcup \{ \dtheta[k][w] )\}$
        \State $t \gets t + \frac{1}{k!} \dtheta[k][w]$
    \EndFor
    \State \Return $t$
\EndProcedure
\end{algorithmic}
\end{algorithm}


\section{Related work}\seclabel{lit_review}
The idea of approximating $\thetaw[w]$ with a Taylor series for the purpose of
CV and the bootstrap has occurred many times in the statistics and machine
learning literature, and we will not attempt a complete survey here.  The
contributions of the present paper are a general treatment of higher-order
expansions of $\thetaw[w]$ with finite-sample bounds that are computable in
principle, together with a description of a practically efficient method for
computing such expansions. Our work builds on \citet{giordano:2019:ij}, which
is restricted to first-order approximations.

Locally linear approximations to CV have attracted much recent interest.
\citet{koh:2017:blackbox} derive a first-order approximation and provide
supporting experiments but little theory.
\citet{beirami:2017:optimalgeneralizability} provide an asymptotic analysis of a
local approximation that is closely related to ours with the notable exception
that they base their approximation on the Hessian matrix $H(\thetaone, w)$
rather than $H(\thetaone, \onevec)$.  Repeatedly calculating $H(\thetaone,
w)^{-1}$ for different weights can incur considerable computational expense
except in special cases.  \citet{beirami:2017:optimalgeneralizability} assumes
that the data domain is bounded (our \condref{supterms_bounded}) and, under this
assumption, derive asymptotic rates that match ours (our
\propref{cv_complexity_uniform}), despite the additional computation cost.
\citet{rad:2018:scalableloo} consider a linear approximation to the optima of
objectives that depend on $\theta$ only through a linear index $\theta^T x_n$
with IID Gaussian data $x_n$.

Some authors have also considered higher-order approximations in special cases.
For example, \citet{liu:2014:efficient, debruyne:2008:model} derive higher-order
expansions for the special case of linear models. \citet{liu:2014:efficient} use
the ``Bouligand derivative'' to consider the robustness of support vector
machines (SVMs) with non-differentiable objectives by
\citet{christmann:2008:bouligand}.  The Bouligand derivative coincides with the
classical Gateaux (directional) derivative when the objective function is
differentiable, and so, when the SVM loss and regularizer are continuous, their
results are a special case of ours.

Forming a local approximation to the bootstrap has been considered by
\citet{schulam:2019:trust} and mentioned by
\citet{beirami:2017:optimalgeneralizability}.  It seems that the importance of
using higher-order approximations for the bootstrap was not recognized by these
authors.  For the purpose of deriving asymptotic results, the connection
between linearity of statistical functionals of empirical distributions and
asymptotic normality is well-known in statistics \citep{mises:1947:asymptotic,
fernholz:1983:mises}.

Our work reposes on the general framework of the differentiability of
statistical functionals of an empirical distribution.  This framework has been
employed in the statistical literature, though typically from an asymptotic
perspective \citep[Chapter 20]{reeds:1976:thesis, fernholz:1983:mises,
van:2000:asymptotic}.  Like the machine learning literature cited above, most
classical theoretical treatments of the functional differentiability of
optimization problems require the objective function to have bounded derivatives
\citep{clarke:1983:uniqueness, shao:2012:jackknife} in order for the functionals
to be Fr{\'e}chet differentiable in a space large enough to contain the
asymptotic limit of empirical processes. Additionally, first-order derivatives
of statistical functionals with respect to an empirical distribution are often
used as a theoretical tool for analyzing other estimators rather than as a
practical computational tool in their own right \citep{mises:1947:asymptotic,
huber:1964:robust, shao:1993:jackknifemestimator} (though exceptions exist,
e.g., \citet{wager:2014:confidence}).  These first-order derivatives are closely
related to both the classical infinitesimal jackknife and the influence
function, differing only in formalities of definition and regularity
assumptions.

Finally, we observe that CV is not the best method for evaluating
out-of-sample error in all cases.  Linear models, in particular,
have a rich literature of alternative methods with better theoretical and
practical properties.  See \citet{efron:2004:estimation, rosset:2018:fixed}
for recent work.

\section{Conclusion}
We describe a higher-order infinitesimal jackknife, analyze its finite-sample
error in terms of practically computable quantities, and provide an algorithm
for computing it efficiently using forward-mode automatic differentiation.
We leave numerical experiments, computable bounds for generalized linear
models, and a more careful treatment of the bootstrap and of k-fold cross
validation for future work.

\paragraph{Acknowledgements.} This research was supported in part by DARPA
(FA8650-18-2-7832), an ARO YIP award, an NSF CAREER award, and the CSAIL-MSR
Trustworthy AI Initiative. Ryan Giordano was supported by the Gordon and Betty
Moore Foundation through Grant GBMF3834 and by the Alfred P. Sloan Foundation
through Grant 2013-10-27 to the University of California, Berkeley. We are
grateful for and indebted to useful discussions with Will Stephenson and Pang
Wei Koh.


\newpage

\bibliography{references}
\bibliographystyle{plainnat}

\newpage
\onecolumn
\appendix

\section{Supporting lemmas}


\begin{lem}\lemlabel{operator_norm_continuity}
Let $A$ and $D$ be two square matrices of the same size.
Let $\norm{A^{-1}}_{op}\le \cop$ for some finite $\cop$, and let $r \in (0, 1)$
be a number strictly between $0$ and $1$. Then
\begin{align*}
\norm{A-D}_{2}\le r \cop^{-1} &
    \quad\Rightarrow\quad\norm{D^{-1}}_{op}\le \frac{1}{1 - r} \cop.
\end{align*}
\begin{proof}
We will use the results stated in Theorem 4.29 of \citet{schott:2016:matrix} and
the associated discussion in Example 4.14, which establish the following result.
Let $A$ be a square, nonsigular matrix, and let $I$ be the identity matrix of
the same size.  Let $\norm{\cdot}$ denote any matrix norm satisfying $\norm{I} =
1$.  Let $B$ be a matrix of the same size as $\norm{A}$ satisfying
\begin{align}\eqlabel{ab_matrix_condition}
\norm{A^{-1}} \norm{B} \le 1.
\end{align}
Then,
\begin{align}\eqlabel{matrix_norm_continuity}
    \norm{A^{-1} - (A + B)^{-1}} \le
    \frac{\norm{A^{-1}}\norm{B}}{1 - \norm{A^{-1}\norm{B}}} \norm{A^{-1}}.
\end{align}
We will apply \eqref{matrix_norm_continuity} using the operator norm
$\norm{\cdot}_{op}$, for which $\norm I_{op}=1$ and with $B := D - A$.
Because $\norm{A^{-1}}_{op}\le \cop$, $A$ is invertible.

Assume that $\norm{A-D}_{2}\le r \cop^{-1}$.  First, note that
\begin{align}\eqlabel{ab_matrix_condition_fulfilled}
\norm{A^{-1}}_{op} \norm{B}_{op} &=
    \norm{A^{-1}}_{op}\norm{D - A}_{op} \nonumber \\
&\le\norm{A^{-1}}_{op}\norm{D - A}_{2}
    & \textrm{(ordering of matrix norms)}\nonumber \\
 & \le \cop r \cop^{-1}
    & \textrm{(by assumption)} \nonumber \\
&= r  < 1,
\end{align}
so \eqref{ab_matrix_condition} is satisfied and we can apply
\eqref{matrix_norm_continuity}. Then
\begin{align*}
\norm{D^{-1}}_{op}
 & \le \norm{D^{-1}-A^{-1}}_{op} + \norm{A^{-1}}_{op}
    & \textrm{ (triangle inequality)}\\
 & \le \frac{\norm{A^{-1}}_{op}\norm{D - A}_{op}}
            {1 - \norm{A^{-1}}_{op}\norm{D - A}_{op}}
        \norm{A^{-1}}_{op} + \norm{A^{-1}}_{op}
     & \textrm{(\eqref{matrix_norm_continuity})}\\
 & \le \frac{r}{1-r}\norm{A^{-1}}_{op} + \norm{A^{-1}}_{op}
    &\textrm{(\eqref{ab_matrix_condition_fulfilled})} \\
 & \le \frac{1}{1-r}\cop.&\textrm{(by assumption)}
\end{align*}
\end{proof}
\end{lem}


\begin{lem}\lemlabel{taylor_series}
%
For a $k_{max}$ times continuously differentiable vector-valued function
$f\left(x\right):\mathbb{R}^{N}\rightarrow\mathbb{R}^{D}$, define the shorthand
notation
\begin{align*}
\fk(x) \deltax^k :=\
    \sum_{n_{1}=1}^{N}\cdots\sum_{n_{k}=1}^{N}
    \partialat{{}^k f(x)}{x_{n_{1}} \ldots \partial x_{n_{k}}}{x}
    \left(x_{1}-x_{0}\right)_{n_{1}}\ldots\left(x_{1}-x_{0}\right)_{n_k},
\end{align*}
taking $\fk[0](x) \deltax^0 = f(x)$.  Fix $x_1$ and $x_0$ and define the
Taylor series residual
\begin{align*}
R_{k_{max}} :=
    f(x_1) -
    \left(\sum_{k=0}^{k_{max} - 1} \frac{1}{k!} \fk(x_0) \deltax^k\right).
\end{align*}

Then
\begin{align*}
R_{k_{max}} =
    \frac{1}{(k_{max} - 1)!}
        \int_{0}^{1}(1-t)^{k_{max}-1} \fk[k_{max}](\xt)\deltax^{k_{max}} dt
\end{align*}
and, for any norm $\norm{\cdot}_p$,
\begin{align*}
\MoveEqLeft
\norm{ R_{k_{max}} }_p \le
\frac{1}{(k_{max} - 1)!} \left(
    \sup_{x: \norm{x-x_{0}}_{\infty}\le\norm{x_{1}-x_{0}}_{\infty}}
        \norm{\fk[k_{max}](x) \deltax^{k_{max}}}_p \right)
\end{align*}
\begin{proof}
Let $f_{d}(x)$ denote the scalar-valued $d$-th component of $f(x)$, so
$f(x)=(f_{1}(x),\ldots,f_{D}(x))^{T}$. Let $\xt := tx_{1}+(1-t)x_{0} = t
\deltax + x_0$. By the chain rule,
\begin{align*}
\derivat{{}^k f_d(\xt)}{t^k}{t} = \fdk(\xt)\deltax^k
\end{align*}
By taking the integral remainder form of the $k_{max}-1$-th order Taylor
series expansion in $t$ of the scalar $f_{d}(\xt)$
\citep[Appendix B.2]{dudley:2018:analysis}, and stacking the result into a
vector,
\begin{align}
R_{k_{max}} =
&\frac{1}{(k_{max} - 1)!}
    \int_{0}^{1}(1-t)^{k_{max}-1} \fk[k_{max}](\xt)\deltax^{k_{max}} dt.
    \eqlabel{taylor_remainder_scalar}
\end{align}
The norm of $R_{k_{max}}$ is bounded by
\begin{align*}
\MoveEqLeft
\norm{R_{k_{max}}}_p & \\
&= \frac{1}{(k_{max} - 1)!} \norm{
    \int_{0}^{1}(1-t)^{k_{max}-1} \fk[k_{max}](\xt)\deltax^{k_{max}} dt
        }_p  & \\
&\le \frac{1}{(k_{max} - 1)!} \int_{0}^{1}(1-t)^{k_{max}-1}
    \norm{\fk[k_{max}](\xt)\deltax^{k_{max}}}_p dt
    & \textrm{(Jensen's inequality)} \\
&\le \frac{1}{(k_{max} - 1)!}
  \left(\sup_{t\in[0, 1]} (1-t)^{k_{max}-1}\right)
  \left(
    \sup_{t\in[0, 1]} \norm{\fk[k_{max}](\xt)\deltax^{k_{max}}}_p
    \right) &\\
&\le \frac{1}{(k_{max} - 1)!}
   \sup_{x: \norm{x-x_{0}}_{\infty}\le\norm{x_{1}-x_{0}}_{\infty}}
    \norm{\fk[k_{max}](x)\deltax^{k_{max}}}_p,
\end{align*}
where the last inequality follows from the fact that
\begin{align*}
\norm{\xt-x_{0}}_{\infty} =
    t \norm{\deltax}_{\infty} \le \norm{\deltax}_{\infty}.
\end{align*}
\end{proof}
\end{lem}

%
%
%

\begin{lem}\proplabel{array_norm_relation} Array norms.
For any norm $\norm{\cdot}_p$, $\norm{\Gk(\theta)}_p \le \frac{1}{N}
\norm{\gk(\theta)}_p$.

\begin{proof}\prooflabel{array_norm_relation}
Let $d$ be a multi-index for $\Gk(\theta)$.  Then
\begin{align*}
\norm{\Gk(\theta)}_p &= \sum_d \norm{\Gk(\theta)_d}_p \\
        &= \sum_d \norm{\frac{1}{N} \sum_{n=1}^N (\gnk(\theta))_d}_p \\
        &\le \sum_d \frac{1}{N} \sum_{n=1}^N \norm{(\gnk(\theta))_d}_p
            & \textrm{(Jensen's inequality)} \\
        &= \frac{1}{N} \norm{\gk(\theta)}_p.
\end{align*}
\end{proof}

\end{lem}


\begin{lem}\lemlabel{hessian_integral_inverse}
%
Suppose that $A(\theta)$ is a matrix such that $\sup_{\theta \in \allthetadom}
\norm{A(\theta)^{-1}}_{op} \le C < \infty$ for some finite $C$.  For example,
\assuref{convexity} fulfills this assumption with $A(\theta) := H(\theta,
\onevec)$ and $C := \cop$.  Then, for any fixed $\theta_0 \in \allthetadom$ and
$\theta_1 \in \allthetadom$,
\begin{align*}
\lambda_{\mathrm{min}}
    \left(\int_{0}^{1} A\left(
        t\theta_1-(1-t)\theta_0\right)dt\right)
        & \ge \frac{1}{C}.
\end{align*}
In particular, the matrix
\begin{align*}
    \int_{0}^{1}A\left(
t\theta_1-(1-t)\theta_0,\onevec\right)dt
\end{align*}
is invertible with
\begin{align*}
\norm{ \int_{0}^{1}A\left(
t\theta_1-(1-t)\theta_0\right)dt }_{op} \le C.
\end{align*}
\begin{proof}
Recall that, for any matrix $A$, $\lambda_{\mathrm{min}}(A) =
\frac{1}{\norm{A^{-1}}_{op}}$.   By definition,
\begin{align*}
\MoveEqLeft
\lambda_{\mathrm{min}}\left(
    \int_{0}^{1}
        A\left( t\theta_1 - (1-t) \theta_0\right) dt\right) \\
 & =\inf_{v:\norm{v}_{2}=1} v^{T}
    \left(\int_{0}^{1}
        A\left( t\theta_1 - (1-t) \theta_0\right) dt\right)v\\
 & = \inf_{v:\norm{v}_{2}=1}\int_{0}^{1} v^{T}
    A\left( t\theta_1 - (1-t) \theta_0\right) v dt\\
 & \ge \inf_{\theta \in \allthetadom}
    \lambda_{\mathrm{min}} (A(\theta))
        & \textrm{(}\allthetadom\textrm{ is convex)}\\
 & = \frac{1}
    {\sup_{\theta\in\allthetadom} \norm{ A(\theta)^{-1}}_{op} } \\
 & \ge \frac{1}{C}. & \textrm{(by assumption)}
\end{align*}
\end{proof}
\end{lem}



\begin{lem}\lemlabel{lipschitz}
%
Under \assuref{bounded},
\begin{align*}
    \sup_{\theta \in \allthetadom}
        \norm{\Hw[\theta, \onevec] - \Hw[\thetaone, \onevec]}_2 \le
        \mk[3] \norm{\theta - \thetaone}_2.
\end{align*}
\end{lem}

\lemref{lipschitz} replaces the explicit assumption of Lipschitz continuity
of the Hessian in \citet{giordano:2019:ij}.


\section{Omitted proofs}

We now present proofs that were omitted from the main manuscript.


\paragraph{Proof of \corref{theta_difference_bound}}
\begin{proof}\prooflabel{theta_difference_bound}
By applying \lemref{taylor_series} with $k_{max} = 1$ to the expansion of
$G\left(\theta, \onevec\right)$ in $\theta$ around $\thetaone$,
\begin{align*}
G\left(\thetaw,\onevec\right) - G\left(\thetaone,\onevec\right) = &
    \int_{0}^{1}
    H\left(t \thetaw - \left(1 - t\right) \thetaone,
           \onevec\right) \left(\thetaw - \thetaone\right)dt  \\
= & \left(\int_{0}^{1}
    H\left(t \thetaw - \left(1 - t\right) \thetaone, \onevec\right) dt \right)
        \left(\thetaw - \thetaone\right)
\end{align*}
Note that $G\left(\thetaone, \onevec\right)=0$ and that
$\thetaw-\hat{\theta}$ can be taken out of the integral as it does not depend on
$t$. By \assuref{convexity} and  \lemref{hessian_integral_inverse}, we can
solve for $\thetaw - \thetaone$ and take the norm of both sides to get
\begin{align*}
\norm{\thetaw - \thetaone}_{2} &
    = \norm{\left(\int_{0}^{1}
     H\left(t\thetaw - (1-t) \thetaone, \onevec\right)dt\right)^{-1}
     G\left(\thetaw, \onevec\right)}_{2} & \\
 & \le\cop\norm{G\left(\thetaw, \onevec\right)}_{2} &
    \textrm{(\lemref{hessian_integral_inverse})}\\
 & =\cop\norm{G\left(\thetaw, \onevec \right)
    - G\left(\thetaw, \wt\right)}_{2} &
    \textrm{(} \thetaw \textrm{ is the estimator for } \wt \textrm{)}\\
 & \le\cop\deltak[0]. &\textrm{(\assuref{complexity})}
\end{align*}
\end{proof}



\paragraph{Proof of \lemref{weighted_hessian_inverse}}
\begin{proof}\prooflabel{weighted_hessian_inverse}
We wish to use \lemref{operator_norm_continuity} to bound
$\sup_{\wt\in\wdom}\norm{\Hw^{-1}}_{op}$ in terms of $\cop =
\sup_{\theta\in\allthetadom}\norm{\Hw[\theta,\onevec]^{-1}}_{op}$. To do
so, we must control $\norm{\Hw[\theta,\wt] - \Hw[\thetaone,\onevec]}_{2}$.
For any $\theta \in \allthetadom$, we have
\begin{align*}
\MoveEqLeft \norm{\Hw[\theta,\wt] - \Hw[\thetaone,\onevec]}_{2} \\
 & = \norm{\Hw[\theta,\wt] - \Hw[\theta,\onevec] +
           \Hw[\theta,\onevec] - \Hw[\thetaone,\onevec]}_{2} &\\
 & \le \norm{\Hw[\theta,\wt] - \Hw[\theta,\onevec]}_{2} +
       \norm{\Hw[\theta,\onevec] - \Hw[\thetaone,\onevec]}_{2}
    & \textrm{ (triangle inequality)}\\
 & \le \deltak[1] +
       \norm{\Hw[\theta,\onevec] - \Hw[\thetaone,\onevec]}_{2}
    & \textrm{(\assuref{complexity})}\\
     & \le \deltak[1]+ \mk[3] \norm{\theta-\thetaone}_{2}.
    & \textrm{(\lemref{lipschitz})}
\end{align*}
Consequently,
\begin{align*}
\MoveEqLeft
\sup_{\wt\in\wdom} \norm{\Hw[\thetaw,\wt] - \Hw[\thetaone,\onevec]}_{1} \\
    &\le \deltak[1] +
         \lh \sup_{\wt\in\wdom} \norm{\thetaw-\thetaone}_{2}
    &\\
 & \le \deltak[1] + \lh\cop\deltak[0].
    & \textrm{(\corref{theta_difference_bound})}
\end{align*}
In order to apply \lemref{operator_norm_continuity}, we require
$\sup_{\wt\in\wdom} \norm{\Hw[\theta,\wt] - \Hw[\thetaone,\onevec]}_{2}
\le \deltabound \cop^{-1}$ for some $0 < \deltabound < 1$.  By the
previous display, this will occur when
\begin{align*}
\deltak[1] + \lh\cop\deltak[0] &\le \deltabound \cop^{-1} \Leftrightarrow \\
\cop \deltak[1] + \cop^2 \lh \deltak[0] &\le \deltabound \Leftrightarrow \\
\cset &\le \deltabound.
\end{align*}
By design, this is precisely what \condref{theta_small} requires.  Thus, under
\condref{theta_small}, by \lemref{operator_norm_continuity},
\begin{align*}
\sup_{\wt\in\wdom} \norm{\Hw[\thetaw, \wt]^{-1}}_{op} &
    \le \frac{1}{1 - \deltabound}\cop.
\end{align*}
The final result follows by multiplying both sides of the left inequality of
the previous display by $\cop$.
\end{proof}



\paragraph{Proof of \lemref{implicit_function_thm}}
\begin{proof}\prooflabel{implicit_function_thm}
The proof uses Theorem 2 of \citet{sandberg:1980:globalinverse} together with a
standard technique relating inverse theorems to implicit function theorems (see,
e.g., Proposition 1 of \citet{sandberg:1980:globalinverse} or Section 3.3 of
\citet{krantz:2012:implicit}).

Define $\gimp(w, \theta): \wdom \times \allthetadom \rightarrow \mathbb{R}^{N +
D}$ by $\gimp(w, \theta) = (w, G(\theta, w))$.  By \assuref{diffable}, $\gimp$
is $\kij + 1$ times continuously differentiable.  Let $\zerovec$ be the
length---$D$ vector of zeroes, and $z \in \mathbb{R}^D$ ($z$ is for ``zero'').
For all $\omega \in \wdom$ wish to consider inverting the equation $\gimp(\wt,
\theta) = (\omega, \zerovec)$, so that (if the inverse exists), $\gimp^{-1}(\wt,
\zerovec) = (\omega, \thetaw)$.  Establishing the existence and
differentiability of $\gimp^{-1}$ will establish \lemref{implicit_function_thm}.

Let $\Omega_z$ be an open neighborhood of $\zerovec$. Theorem 2 of
\citet{sandberg:1980:globalinverse} states that the inverse $\gimp^{-1}$ exists
and is $\kij + 1$ times continuously differentiable on $\wdom \times \zdom$ if,
for all $\omega \in \wdom$, $\theta \in \allthetadom$, and $z \in \zdom$,

\begin{enumerate}
\item $\wdom \times \allthetadom$ is an open set homeomorphic to
      $\mathbb{R}^{N + D}$,
\item $\mathrm{det}
    \left(\partialat{\gimp(w, \theta)}{(w, \theta)}{\wt, \theta}
        \right) \ne 0$, and
\item
For $\omega_0, \omega_z \in \wdom$ and $\theta_0, \theta_z \in \allthetadom$,
if $\gimp(\wt_0, \theta_0) = (\omega_0, \zerovec)$,
$\gimp(\wt_z, \theta_z) = (\omega_z, z)$, and
$(\omega_z, z)$ is close to $(\omega_0, \zerovec)$, then $(\wt_z, \theta_z)$ is
close to $(\wt_0, \theta_0)$.  (See below for a precise statement.)
\end{enumerate}

The first condition is satisfied by \defref{wdom} and \assuref{thetadom}.
Because $\partialat{\gimp(w, \theta)}{w}{\wt_0, \theta_0} = (I_N, 0)$, where
$I_N$ is the $N\times N$ identity matrix, the second condition is equivalent to
$\mathrm{det} \left(\partialat{G(\theta, w)}{\theta}{\wt_0, \theta_0}\right) \ne
0$, which is implied by \lemref{weighted_hessian_inverse}.

The final condition amounts to the requirement that $\gimp(\wt, \theta)$ is a
proper map.  Specifically, we need to show that, for any $\epsilon > 0$, if
$\norm{ (\omega_z, z) - (\omega_0, \zerovec)}_2 = \norm{\omega_z - \omega_0}_2 +
\norm{z - \zerovec}_2\le \epsilon$, then there exists a $\delta > 0$ such that
$\norm{ (\wt_z, \theta_z) - (\wt_0, \theta_0)}_2 = \norm{\wt_z - \wt_0}_2 +
\norm{\theta_z - \theta_0}_2\le \delta$.

By applying \lemref{taylor_series}  with $k_{max} = 1$ to the expansion of
$G(\theta, \wt)$ in $\theta$ around $\theta_0$ at $\wt_0$,
\begin{align}\eqlabel{implicit_fn_integral}
G(\theta_z, w_z) - G(\theta_0, w_0) \nonumber
= & z - \zerovec \\
= & \left(\int_{0}^{1}
    H(t \theta_z - (1 - t) \theta_0, \wt_0) dt \right) (\theta_z - \theta_0).
\end{align}
By \lemref{weighted_hessian_inverse} and \lemref{hessian_integral_inverse},
we can solve for $\theta_z - \theta_0$ and take norms of both sides to get
\begin{align*}
\norm{\theta_z - \theta_0}_2
\le & \copw \norm{z - \zerovec}_2
    & \textrm{(\lemref{weighted_hessian_inverse} and
               \lemref{hessian_integral_inverse} )}\\
\le & \copw \epsilon.
    & \textrm{(by assumption and the triangle inequality)}
\end{align*}
Consequently, taking $\delta = (1 + \copw) \epsilon$ gives
$\norm{\wt_z - \wt_0}_2 + \norm{\theta_z - \theta_0}_2 \le \delta$, completing
the proof.

Intuitively, \eqref{implicit_fn_integral} shows that, if $\norm{\theta_z -
\theta_0}_2$ is to be large, then the Hessian matrix $H(t \theta_z - (1 - t)
\theta_0, \wt_0)$ must be small somewhere on the line between $\theta_z$ and
$\theta_0$.  This is prohibited, however, by \assuref{convexity}, which lower
bounds the minimum eigenvalue of $H(\theta, \onevec)$ for all $\theta \in
\allthetadom$, and by \lemref{weighted_hessian_inverse} (depending on the set
complexity bound of \condref{theta_small}), which asserts that the smallest
eigenvalue of of re-weighted versions of $H(\theta, w)$ are never too much
smaller than the smallest eigenvalue of $H(\theta, \onevec)$.

\end{proof}



\paragraph{Proof of \lemref{higher_order_general_terms}}
\begin{proof}\prooflabel{higher_order_general_terms}

The proof will proceed by induction on $k$.
The inductive assumption will be that, for some $k \ge 1$, then the
$k$-th order directional derivative of $G(\theta, w)$ is of the form
\begin{align}\eqlabel{dgk_form_inductive_assumption}
\derivat{{}^{k} G(\thetaw[w], w)}{w \ldots dw}{\wt} \deltaw \ldots \deltaw & =
    \dterm[\{ k \}][0][\wt] +
    \sum_{(a_i, \kvec_i, \omega_i) \in \dtermargs[k]}
            \dterm[\kvec_{i}][\omega_i][\wt],
\end{align}
where elements of $\dtermargs[k]$ satisfy the conditions
\eqref{dtheta_term_properties} and \eqref{dtheta_term_total_order}.

Intuitively, \eqref{dgk_form_inductive_assumption} states that there is only one
way a term containing $\dtheta[k]$ can arise in $\derivat{{}^{k} G(\thetaw[w],
w)}{w \ldots dw}{\wt} \deltaw \ldots \deltaw$, and that is by differentiating
the $\dtheta[k-1]$ part of a term of the form $\dterm[\{k-1\}][0][\wt]$. The
property \eqref{dtheta_term_properties} then follows from solving for
$\dtheta[k]$ in said term, and \eqref{dtheta_term_total_order} states the
intuitively obvious fact that, after differentiating $G(\thetaw, w)$ $k$ times,
each resulting term must contain a total of precisely $k$ directional
derivatives.  The following proof is simply a formalization of this intuition.

The case $k=1$ follows from \eqref{dgw_term_form}, which we reproduce here for
convenience.
\begin{align}\eqlabel{dgw_term_form_appendix}
\derivat{G(\thetaw[w], w)}{w}{\wt} \deltaw &=
    \dterm[\{1\}][0][\wt] + \dterm[\emptyset][1][\wt].
\end{align}
Recall that \eqref{dgw_term_form} is simply a rewrite of \eqref{dg1}. Here,
$\dtermargs[1] = \{ (a_1=1, \kvec_1=\emptyset, \omega_1=1) \}$, corresponding to
the second term in \eqref{dgw_term_form_appendix}.  The first term in
\eqref{dgw_term_form_appendix}, $\dterm[\{1\}][0][\wt]$, is of the required form
$\dterm[\{k\}][0][\wt]$ for $k=1$.

Next, we show that if \eqref{dgw_term_form} holds for $k$, it holds for $k + 1$.
By the chain rule, the directional derivative of a generic term
$\dterm[\kvec][\omega][\wt]$ in the direction $\deltaw$ is given by
\begin{align}
\MoveEqLeft
\derivat{\dterm[\kvec][\omega][w]}{w}{\wt}\deltaw  = & \nonumber \\
& \sum_{k\in\kvec} \dterm[\kvec\setminus \{ k \} \bigcup \{k + 1\}][\omega][\wt] +
    &\textrm{(} \dtheta[k] \textrm{ terms)}
    \eqlabel{differentiation_term_dtheta} \\
& \dterm[\kvec \bigcup \{1\}][\omega][\wt] +
    &\textrm{(the first argument of } \Gk[|\kvec|](\thetaw, \wt)\textrm{)}
    \eqlabel{differentiation_term_dgfirst}\\
&(1 - \omega) \dterm[\kvec][1][\wt].
    &\textrm{(the second argument of } \Gk[|\kvec|](\thetaw, \wt)\textrm{)}
    \eqlabel{differentiation_term_dgsecond}
\end{align}

We first show that, if $\dterm[\kvec][\omega][\wt]$ satisfies
\eqref{dtheta_term_properties, dtheta_term_total_order} for $k \ge 1$, then
$\derivat{\dterm[\kvec][\omega][w]}{w}{\wt}\deltaw$ satisfies
\eqref{dtheta_term_properties, dtheta_term_total_order} for $k + 1$. First
of all, by \eqref{dtheta_term_properties} each element of $\kvec$ is at most
$k$ and increases by at most $1$, so, after differentiation, no elements of any
$\kvec$ exceed $k + 1$, satisfying \eqref{dtheta_term_properties} for $k + 1$.
Second, for every term in \eqref{differentiation_term_dtheta} and
\eqref{differentiation_term_dgfirst}, the sum of elements of $\kvec$ and
$\omega$ increases by exactly one to $k + 1$, satisfying
\eqref{dtheta_term_total_order}. Finally, for the final term in
\eqref{differentiation_term_dgsecond}, the sum of the elements of $\kvec$
remains $k$, but we have either changed $\omega$ from $0$ to $1$ or eliminated
the term because $\omega = 1$. Consequently,
\eqref{dtheta_term_properties, dtheta_term_total_order} are satisfied for
each term in $\derivat{\dterm[\kvec][\omega][w]}{w}{\wt} \deltaw$.

Next, if $\dterm[\kvec][\omega][\wt]=\dterm[\{k\}][0][\wt]$,
then exactly one term arises from the sum in
\eqref{differentiation_term_dtheta}, and that is $\dterm[\{k+1\}][0][\wt]$.
The remaining terms from \eqref{differentiation_term_dgfirst} satisfy
\eqref{dtheta_term_properties, dtheta_term_total_order} by the
same argument as the preceding paragraph.

Finall, by induction, \eqref{dgw_term_form} holds for all $k\ge1$. The constants
$a_{i}$ in $\dtermargs[k+1]$ will arise from the corresponding coefficients in
$\dtermargs[k]$ and grouping equivalent terms.

Because $\dterm[\{k\}][\wt][\wt] = \Hw \dtheta[k]$ and, by
\lemref{weighted_hessian_inverse}, $\Hw$ is invertible, we can solve
\eqref{dgw_term_form} for $\dtheta[k]$:
\begin{align*}
\dtheta[k] &= -\Hw^{-1}
\sum_{(a_i, \kvec_i, \omega_i) \in \dtermargs[k]}
    a_i \dterm[\kvec_{i}][\omega_i][\wt]
\end{align*}
which is the desired result.
\end{proof}



\paragraph{Proof of \lemref{taylor_theta_term_bounds}}
\begin{proof}\prooflabel{taylor_theta_term_bounds}

Applying Cauchy-Schwartz and the relationship between the L1 and L2 norms
inductively,
\begin{align*}
\MoveEqLeft
\norm{  \Gk[K](\thetaw, \wt) \dtheta[k_1] \ldots \dtheta[k_K]    }_1  \\
& \le   \norm{  \Gk[K](\thetaw, \wt) \dtheta[k_1]\ldots\dtheta[k_{K-1}]    }_2
    \norm{  \dtheta[k_K] }_2
    & \textrm{(Cauchy-Schwartz)}\\
&\le   \norm{  \Gk[K](\thetaw, \wt) \dtheta[k_1]\ldots\dtheta[k_{K-1}]    }_1
    \norm{  \dtheta[k_K] }_2
    & \textrm{(relationship between norms)}\\
&\le   \norm{  \Gk[K](\thetaw, \wt)   }_2 \prod_{i=1}^{K}\norm{  \dtheta[k_i] }_2.
    & \textrm{(induction on }k\textrm{)}
\end{align*}
Next,
\begin{align*}
\MoveEqLeft
\norm{  \Gk[K](\thetaw, \wt)    }_1\\
&  \le   \norm{  \Gk[K](\thetaw, \wt) - \Gk[K](\thetaw, \onevec)    }_2 +
    \norm{  \Gk[K](\thetaw, \onevec)    }_2
    & \textrm{(triangle inequality)} \\
&  \le  \deltak[K] + \norm{  \Gk[K](\thetaw, \onevec)    }_2
    & \textrm{(\assuref{complexity})} \\
&   \le \deltak[K] + \mk[K]. & \textrm{\assuref{bounded}}
\end{align*}
Following the same reasoning, we have
\begin{align*}
\MoveEqLeft
\norm{  \Gk[K](\thetaw, \deltaw)\dtheta[k_1]\ldots\dtheta[k_K]    }_1 \\
& \le   \norm{  \Gk[K](\thetaw, \deltaw)   }_2
    \prod_{i=1}^{K}\norm{  \dtheta[k_i] }_2 & \\
& \le   \deltak[K] \prod_{i=1}^{K}\norm{  \dtheta[k_i] }_2. &
    \textrm{(\assuref{complexity})}
\end{align*}
The desired result follows.
\end{proof}



\paragraph{Proof of \thmref{thetaij_consistent}}
\begin{proof}\prooflabel{thetaij_consistent}
The proof proceeds by induction on $k$.  \Thmref{thetaij_consistent} holds
for $k=1$ by \corref{d2theta_bounded} and \corref{thetaij1_consistent}.

Suppose that \thmref{thetaij_consistent} holds for $1 \le k < \kij$.  Then
\begin{align*}
\MoveEqLeft
\norm{\dtheta[k+1]}_2 = \norm{\Hw^{-1}
    \left(\sum_{i}a_{i} \dterm[\kvec_i][\omega_i][\wt] }_2 \right)
    &\textrm{(\lemref{higher_order_general_terms})}&\\
&\le \copw \norm{\sum_{i}a_{i}\dterm[\kvec_i][\omega_i][\wt]}_2
   &\textrm{(\lemref{weighted_hessian_inverse})} \\
&\le \copw \sum_{i}a_{i}\norm{\dterm[\kvec_i][\omega_i][\wt]}_2
    &\textrm{(triangle inequality)}\\
& \le \copw
    \sum_{(a_i, \kvec_i, \omega_i) \in \dtermargs[k]}
    a_{i} (\deltak[|\kvec_i|] + (1 - \omega_i) \mk[|\kvec_i|])
        \prod_{\tilde{k} \in \kvec_i} \norm{\dtheta[\tilde{k}]}_2
    &\textrm{(\lemref{taylor_theta_term_bounds}, relation between norms)}\\
& =
    \sum_{(a_i, \kvec_i, \omega_i) \in \dtermargs[k]}
    (\deltak[|\kvec_i|] + (1 - \omega_i) \mk[|\kvec_i|])
        \prod_{\tilde{k} \in \kvec_i} O\left(\deltamax^{\tilde{k}}\right)
    &\textrm{(inductive assumption)}\\
& =
    \sum_{(a_i, \kvec_i, \omega_i) \in \dtermargs[k]}
    O\left( \deltamax^{\omega_i + \sum_{\tilde{k} \in \kvec_i}\tilde{k}}\right)
    & \textrm{(}\mk[|\kvec_i|]=O(1)\textrm{ by \assuref{bounded})}\\
& = O\left(\deltamax^{k+1} \right).
    &\textrm{(\lemref{higher_order_general_terms})}\\
\end{align*}
Then $\norm{\thetaij - \thetaw}_2 = O\left( \deltamax^{k + 1} \right)$ follows
from \lemref{taylor_series}.
\end{proof}



%
\paragraph{Proof of \propref{linear_bootstrap}}
\begin{proof}\prooflabel{linear_bootstrap}
Let $I_N$ denote the $N$-dimensional identity matrix.  By standard properties
of the multinomial distribution,
$\cov_b(w_b) = I_N - \frac{1}{N}\onevec \onevec^T$.  Consequently,
\begin{align*}
\cov_b\left(\thetaij[1][w_b]\right) &=
    \cov_b\left(-\Hone^{-1} G(\thetaone, w_b - \onevec)\right) \\
&= \Hone^{-1} (\gk[0])^T \cov_b\left(w_b\right) \gk[0] \Hone^{-1} \\
&= \Hone^{-1} (\gk[0])^T
    \left(I_N - \frac{1}{N}\onevec \onevec^T\right) \gk[0] \Hone^{-1} \\
&= \Hone^{-1} \widehat{\cov}(g_n(\thetaone)) \Hone^{-1}.
\end{align*}
\end{proof}


\end{document}